\documentclass[10pt,a4paper]{article}

\usepackage{amsmath,amssymb,bbm}
\usepackage{theorem,ifthen}

\newtheorem{prop}{Proposition}[section]
\newtheorem{lemma}[prop]{Lemma}
\newtheorem{theo}[prop]{Theorem}
\newtheorem{coroll}[prop]{Corollary}
{\theorembodyfont{\rmfamily}
 \newtheorem{remark}[prop]{Remark}
 \newtheorem{definition}[prop]{Definition}
 \newtheorem{example}[prop]{Example}
}
\newcommand{\qedup}{\par\vspace{-1.5em}}

\newenvironment{proof}[1][Proof]{%
  \begin{list}{}{%
      \settowidth{\labelwidth}{\textit{#1.}}%
      \setlength{\itemindent}{\labelwidth}%
      \addtolength{\itemindent}{\labelsep}%
      \setlength{\leftmargin}{0pt}
      \setlength{\parsep}{0pt}
      \setlength{\listparindent}{\parindent}
    }\item[\textit{#1.}]}
  {\hspace*{\fill}$\Box$ \end{list}}
\newenvironment{proof*}[1][Proof]{%
  \begin{list}{}{%
      \settowidth{\labelwidth}{\textit{#1.}}%
      \setlength{\itemindent}{\labelwidth}%
      \addtolength{\itemindent}{\labelsep}%
      \setlength{\leftmargin}{0pt}
    }\item[\textit{#1.}]}
  {\end{list}}

\newcommand{\marginpostpar}[1]{\marginpar{}}
\newcommand{\marginnote}[1]{\marginpar{}}

\newcommand{\braket}[2]{\langle #1|#2\rangle}
\newcommand{\Braket}[2]{\Bigl\langle #1\Big|#2\Bigr\rangle}
\newcommand{\Jorth}{{\langle\perp\rangle}}

\newcommand{\Jadj}{{\langle*\rangle}}
\newcommand{\pmat}[1]{\begin{pmatrix}#1\end{pmatrix}}
\newcommand{\set}[2]{\{#1\,|\,#2\}}
\newcommand{\Set}[2]{\bigl\{#1\,\big|\,#2\bigr\}}

\DeclareMathOperator{\mspan}{span}
\DeclareMathOperator{\Real}{Re}
\DeclareMathOperator{\Imag}{Im}

\newcommand{\proj}{\mathrm{pr}}
\newcommand{\bbN}{\mathbbm{N}}
\newcommand{\bbZ}{\mathbbm{Z}}
\newcommand{\bbR}{\mathbbm{R}}
\newcommand{\bbC}{\mathbbm{C}}
\newcommand{\mdef}{\mathcal{D}}
\newcommand{\range}{\mathcal{R}}

\newcommand{\rsub}{\mathcal{L}}
\newcommand{\rbsdec}{\sideset{}{^2}\bigoplus}

\newlength{\scriptpoint}
\newcommand{\sddots}{
  \raisebox{2\scriptpoint}{$\scriptstyle.$}
  \raisebox{\scriptpoint}{$\scriptstyle.$}
  \makebox[\scriptpoint]{$\scriptstyle.$}
}

\begin{document}

\title{Hamiltonians with Riesz Bases of Generalised Eigenvectors
  and Riccati Equations
}
\author{Christian Wyss\footnote{
    Department of
    Mathematics and Informatics,
    University of Wuppertal,
    Gau\ss stra\ss e 20,
    D-42119 Wuppertal,
    Germany,
    \texttt{wyss@math.uni-wuppertal.de}}}
\date{\today}
\maketitle

\begin{center}
  \parbox{11cm}{\small \textbf{Abstract.}
    An algebraic  Riccati equation for linear operators is studied,
    which arises in systems theory.
    For the case that all involved operators are unbounded, 
    the existence of 
    infinitely many selfadjoint solutions 
    is shown.
    To this end,
    invariant graph subspaces
    of the associated Hamiltonian operator matrix
    are constructed
    by means of
    a Riesz basis with parentheses of generalised 
    eigenvectors and two indefinite inner products.
    Under additional assumptions, the existence and a representation of
    all bounded solutions is obtained.
    The theory is applied to Riccati equations of differential operators.
  } \\[1.5ex]
  \parbox{11cm}{\small\textbf{Keywords.}
    Riccati equation, Hamiltonian operator matrix,
    Riesz basis of generalised eigenvectors,
    invariant subspace, indefinite inner product.
  } \\[1.5ex]
  \parbox{11cm}{\small\textbf{Mathematics Subject Classification.}  
    Primary 47A62; Secondary 47A15, 47A70, 47B50, 47N70.}
\end{center}
\vspace{0em}

\section{Introduction}

We consider the algebraic Riccati equation
\begin{equation}\label{eq:are}
  A^*X+XA+XBX-C=0
\end{equation}
for linear operators on a Hilbert space $H$ where
$B,C$ are selfadjoint and nonnegative.
In particular, we study the case where $B$ and $C$ are unbounded.
Riccati equations of type \eqref{eq:are} are a key tool in systems theory,
see e.g.\ \cite{curtain-zwart,lancaster-rodman} and the references therein.
Unbounded $B$ and $C$ appear e.g.\ in
\cite{lasiecka-triggiani,pritchard-salamon,
  weiss-weiss}.

It is well known that solutions $X$ of \eqref{eq:are} are in
one-to-one correspondence with graph subspaces which are invariant
under the operator matrix
\[T=\pmat{A&B\\C&-A^*},\]
the  so-called Hamiltonian.
This correspondence was extensively studied in the finite-dimensional setting
and  led to a
complete description of all solutions of the Riccati equation,
see e.g.\ \cite{lancaster-rodman,martensson,potter}.
In the infinite-dimensional setting with $B,C$ bounded, 
the  invariant subspace approach was used 
by Kuiper and Zwart \cite{kuiper-zwart} for Riesz-spectral $T$
and by Langer, Ran and van de Rotten \cite{langer-ran-rotten}
for dichotomous $T$ (see also \cite{bubak-mee-ran}).

We extend these results to the case where $B$ and $C$ are unbounded:
For Hamiltonians with a
Riesz basis with parentheses of generalised eigenvectors
we show the existence of infinitely many 
selfadjoint solutions of \eqref{eq:are}.
Note that the concept of a Riesz basis with parentheses of generalised
eigenvectors includes Riesz-spectral operators,
and it also allows for operators which are not dichotomous.

In systems theory, solutions of \eqref{eq:are}
which are bounded and nonnegative are of
particular importance.
For the case that $T$ has a Riesz basis of generalised eigenvectors,
that its spectrum is contained in a strip around the imaginary axis,
and that $B$ and $C$ are uniformly positive, we prove that there are
infinitely many bounded,
selfadjoint, boundedly invertible solutions,
among them a nonnegative one $X_+$ and a nonpositive one $X_-$.
Moreover, for every bounded selfadjoint solution $X$ we prove the relations
\begin{equation}\label{eq:solrel}
  X_-\leq X\leq X_+ \qquad\text{and}\qquad X=X_+P+X_-(I-P),
\end{equation}
where $P$ is an appropriate projection.

Bounded nonnegative solutions of \eqref{eq:are} were obtained 
in \cite{kuiper-zwart,langer-ran-rotten}
without  the assumption of uniform positivity of $B,C$.
However, in \cite{langer-ran-rotten} the spectrum $\sigma(A)$ of $A$ was 
restricted to a sector in the open left half-plane
while here $\sigma(A)$ may also contain points in the closed right
half-plane.
In \cite{kuiper-zwart} conditions for the existence of solutions were
formulated in terms of the eigenvectors of $T$ while we impose
conditions on the operators $A,B,C$ only.
In the system theoretic setting, the relations \eqref{eq:solrel} were derived
in \cite{curtain-iftime-zwart,opmeer-iftime},
yet under the explicit assumption of the
existence of $X_-$.

For general block operator matrices, invariant graph subspaces 
are connected to solutions of a corresponding Riccati equation too.
This was exploited in \cite{langer-tretter,ran-mee}
for certain dichotomous operator matrices
and in \cite{kostrykin-makarov-motovilov} for selfadjoint ones.
We also mention that, 
in systems theory,
nonnegative solutions of
\eqref{eq:are} are constructed by minimising 
a quadratic functional, see e.g.\ \cite{curtain-zwart}.

The structure of this article is as follows:
In Sections~\ref{sec:rbs} and~\ref{sec:invrbs} we study
the concept of a Riesz basis of subspaces which is  finitely spectral
for a linear operator on a Hilbert space.
Such a Riesz basis consists of finite-dimensional invariant subspaces, 
and it yields many non-trivial
infinite-dimensional invariant subspaces, which we call compatible,
see Corollary~\ref{coroll:invcomp}.
Up to certain technical details, a finitely spectral Riesz basis of subspaces
is equivalent to
a Riesz basis with parentheses of generalised eigenvectors,
see 
Remark~\ref{rem:invrbs-relat}.
Here we use the basis of subspaces notion 
since it  is more convenient for our purposes.
For the relation to dichotomous operators, see 
Remark~\ref{rem:dichot}.

In Theorem~\ref{theo:psubpert-invrbs}
we use perturbation theory to prove a general existence result for
finitely spectral Riesz bases of subspaces
and apply it
to Hamiltonians in Theorem~\ref{theo:ham-invrbs};
Theorem~\ref{theo:ham-rbeigvec} even yields a Riesz basis of eigenvectors and 
finitely many generalised eigenvectors.
On the other hand, there is a huge literature on Riesz bases
(with or without parentheses) of eigenvectors
for various types of operators, 
e.g.\ \cite{jacob-trunk-winklmeier,xu-yung,zwart};
all these provide examples for finitely spectral
Riesz bases of subspaces.

In Section~\ref{sec:ham} we use ideas from \cite{langer-ran-rotten}
and consider
two indefinite inner products with fundamental symmetries
$J_1$ and $J_2$ which are associated with the Hamiltonian:
$T$ is $J_1$-skew-symmetric and $J_2$-accretive.
This implies the symmetry of the spectrum of $T$ with respect to the
imaginary axis and also yields a characterisation of the purely imaginary
eigenvalues.
In Section~\ref{sec:invsubham} we then construct hypermaximal $J_1$-neutral
as well as $J_2$-nonnegative and -nonpositive compatible
subspaces; see Theorem~\ref{theo:hypmaxinv} and
Proposition~\ref{prop:posinv}.

The main existence theorems for solutions of \eqref{eq:are} are presented
in Sections~\ref{sec:sol} and~\ref{sec:bndsol}:
In Theorem~\ref{theo:riccsol} we establish conditions on $T$ such that every
hypermaximal $J_1$-neutral compatible subspace is the graph of a selfadjoint
solution of \eqref{eq:are}.
$J_2$-nonnegative and -nonpositive subspaces yield nonnegative and
nonpositive solutions.
Corollary~\ref{coroll:nonnegsol} provides a sufficient condition
for the existence of
infinitely many selfadjoint solutions.
All these solutions are unbounded in general, and therefore
the Riccati equation takes a slightly different form, see also
Proposition~\ref{prop:riccequiv} and Example~\ref{ex:unbndsol}.
The existence of bounded solutions and the relations \eqref{eq:solrel}
are proved in Theorem~\ref{theo:bndriccsol}.
Finally note that the graph of an arbitrary solution of \eqref{eq:are}
is $T$-invariant, but not necessarily a compatible
subspace, compare Theorem~\ref{theo:bndsol-charac} and
Example~\ref{ex:req-uposmult}.

\section{Riesz bases of subspaces}
\label{sec:rbs}

We recall the closely related concepts of Riesz bases,
Riesz bases with parentheses, and Riesz bases of subspaces, see
\cite[\S 1]{vizitei-markus}, \cite[Chapter VI]{gohberg-krein},
\cite[\S15]{singer2} and \cite[\S 2]{wyss-phd} for more details.

Let $V$ be a separable Hilbert space.
We denote the subspace generated by a family 
$(V_\lambda)_{\lambda\in\Lambda}$ of subspaces $V_\lambda\subset V$ by
\[\sum_{\lambda\in\Lambda}V_\lambda=\{x_{\lambda_1}+\dots+x_{\lambda_n}
  \,|\,x_{\lambda_j}\in V_{\lambda_j},\lambda_j\in\Lambda, n\in\bbN\}.\]
The family is said to be \emph{complete} if 
$\sum_{\lambda\in\Lambda}V_\lambda\subset V$ is dense.

\begin{definition}\label{def:rbs}
  Let $V$ be a separable Hilbert space.
  \begin{itemize}
  \item[(i)]
    A sequence $(v_k)_{k\in\bbN}$ in $V$ is called a \emph{Riesz basis}
    of $V$ if there is an isomorphism $\Phi:V\to V$ such that
    $(\Phi v_k)_{k\in\bbN}$ is an orthonormal basis of $V$.
  \item[(ii)]
    A sequence of closed subspaces $(V_k)_{k\in\bbN}$ of $V$ is called a
    \emph{Riesz basis of subspaces} of $V$ if there is an isomorphism
    $\Phi:V\to V$ such that $(\Phi(V_k))_{k\in\bbN}$ is a
    complete system  of pairwise orthogonal subspaces.
  \end{itemize}
\end{definition}

The sequence $(v_k)_{k\in\bbN}$ is a Riesz basis if and only if
$\mspan\{v_k\}\subset V$ is dense and there are constants $m,M>0$ such that
\begin{equation}\label{eq:rb-ineq}
  m\sum_{k=0}^n|\alpha_k|^2\leq\biggl\|\sum_{k=0}^n\alpha_kv_k\biggr\|^2\leq
  M\sum_{k=0}^n|\alpha_k|^2, \quad \alpha_k\in\bbC,\,n\in\bbN.
\end{equation}
In this case
every $x\in V$ has a unique representation
$x=\sum_{k=0}^\infty\alpha_k v_k$, $\alpha_k\in\bbC$,
where the convergence of the series is unconditional.
The sequence of closed subspaces $(V_k)_{k\in\bbN}$ is
a Riesz basis of subspaces of $V$ if and only if
$(V_k)_{k\in\bbN}$ is complete and
there exists a constant $c\geq 1$ such that
\begin{equation}\label{rbs-ineq-fin}
  c^{-1}\sum_{k\in F}\|x_k\|^2 \leq \Big\|\sum_{k\in F}x_k\Big\|^2
  \leq c\sum_{k\in F}\|x_k\|^2
\end{equation}
for all finite subsets $F\subset\bbN$ and $x_k\in V_k$.

\begin{prop}\label{prop:rbs}
  A Riesz basis of subspaces $(V_k)_{k\in \bbN}$ has the following
  properties:
  \begin{itemize}
  \item[(i)] There are projections $P_k\in L(V)$ onto $V_k$
    satisfying $P_jP_k=0$ for $j\neq k$ and a constant $c\geq1$ such that
    \begin{equation}\label{rbs-ineq}
      c^{-1}\sum_{k=0}^\infty\|P_k x\|^2\leq\|x\|^2
      \leq c\sum_{k=0}^\infty\|P_k x\|^2
      \quad\text{for all}\quad x\in V.
    \end{equation}
  \item[(ii)]
    If\, $x_k\in V_k$ with 
    $\sum_{k=0}^\infty\|x_k\|^2<\infty$, then
    the series $\sum_{k=0}^\infty x_k$ converges unconditionally.
  \item[(iii)] Every $x\in V$ has a unique expansion
    $x=\sum_{k=0}^\infty x_k$ with $x_k\in V_k$,
    and we have $x_k=P_k x$.
  \end{itemize}
\end{prop}
\begin{proof}
  The proof is immediate since all assertions hold (with $c=1$) if
  the $V_k$ are pairwise orthogonal, and they continue to hold
  (with some $c\geq1$ now) if we apply the isomorphism $\Phi$ from
  Definition~\ref{def:rbs}.
\end{proof}

For a Riesz basis of subspaces $(V_k)_{k\in\bbN}$,
the unique expansion from (iii) yields a decomposition of
the space $V$ into the subspaces $V_k$, which we denote by
\begin{equation}\label{rbs-dec}
  V=\rbsdec_{k\in \bbN}V_k.
\end{equation}
Here, the superscript $2$ indicates that, due to \eqref{rbs-ineq}, the 
original norm on $V$ is 
equivalent to the $l^2$-type norm $(\sum_{k\in\bbN}\|P_k x\|^2)^{1/2}$.

Consider now closed subspaces
$U_k\subset V_k$. Then evidently $(U_k)_{k\in\bbN}$ is a Riesz basis of
subspaces of the closed subspace generated by the $U_k$, i.e.\
\[\overline{\sum_{k\in\bbN}U_k}=\rbsdec_{k\in\bbN}U_k.\]
Analogously, for every $J\subset\bbN$ we have that $(V_k)_{k\in J}$ is a
Riesz basis of subspaces of $\bigoplus^2_{k\in J}V_k$.\footnote{
  Note here that Definition~\ref{def:rbs} implicitly covers the case
  of families with arbitrary index set $J\subset\bbN$ since $V_k=\{0\}$ is
  possible.}

\begin{definition}
  Let $(V_k)_{k\in\bbN}$ be a Riesz basis of subspaces of $V$.
  We say that a subspace $U\subset V$ is \emph{compatible} with
  $(V_k)_{k\in\bbN}$ if 
  \[U=\sideset{}{^2}\bigoplus_{k\in\bbN}U_k\quad\text{with closed subspaces }
    U_k\subset V_k.\]
  \qedup
\end{definition}

It is easy to see that, with $P_k$ as above, $U$ is compatible with
$(V_k)$ if and only if
$P_k(U)\subset U$; in this case $U=\bigoplus^2_kP_k(U)$.

If $U$ and $W$ are two subspaces of $V$ satisfying $U\cap W=\{0\}$,
we say that their sum is \emph{algebraic direct}, denoted by
$U\dotplus W$. We say that the sum is \emph{topological direct} and write
$U\oplus W$ if the associated projection from $U\dotplus W$ onto $U$ is
bounded. 
By the closed graph theorem,
if $U\cap W=\{0\}$ and $U$, $W$ and $U\dotplus W$ are closed, 
then in fact $U\oplus W$  is topological direct.

\begin{prop}\label{prop:rbs-decomp}
  Let $(V_k)_{k\in\bbN}$ be a Riesz basis of subspaces.
  \begin{itemize}
  \item[(i)] If\, $V_k=U_k\oplus W_k$ for all $k$, then the sum
    \[\rbsdec_{k\in\bbN} U_k\dotplus\rbsdec_{k\in\bbN} W_k\subset V\]
    is algebraic direct and dense.
  \item[(ii)] For $J\subset\bbN$ we have the topological direct sum
    \[V=\rbsdec_{k\in J}V_k\oplus\rbsdec_{k\in\bbN\setminus J}V_k.\]
    The associated projection onto the first component is given by
    \begin{equation}\label{rbs-proj}
      P_J:\sum_{k\in \bbN}x_k\mapsto\sum_{k\in J}x_k,
      \quad x_k\in V_k,
    \end{equation}
    and satisfies $\|P_J\|\leq c$, where $c$ is the constant from
    \eqref{rbs-ineq}.
  \end{itemize}
\end{prop}
\begin{proof}
  (i): Let $U=\bigoplus^2_{k}U_k$, $W=\bigoplus^2_{k}W_k$,
  and $x\in U\cap W$. We expand $x$ in the Riesz bases $(U_k)$
  of $U$ and $(W_k)$ of $W$: $x=\sum_ku_k=\sum_kw_k$
  with $u_k\in U_k$, $w_k\in W_k$. 
  As these are also expansions of $x$ in the Riesz basis
  $(V_k)$, we obtain $u_k=w_k$ and thus $u_k=0$ and $x=0$.
  The sum $U+W$ is dense in $V$ since it contains $\sum_{k\in\bbN}V_k$.

  (ii): From \eqref{rbs-ineq} we have the estimate
  \[\Big\|\sum_{k\in J}x_k\Big\|^2 \leq c\sum_{k\in J}\|x_k\|^2
    \leq c\sum_{k\in\bbN}\|x_k\|^2 \leq c^2\Big\|\sum_{k\in\bbN}x_k\Big\|^2.\]
  This shows that $P_J$ defined by \eqref{rbs-proj}  satisfies
  $\|P_J\|\leq c$. Obviously 
  \[\range(P_J)=\rbsdec_{k\in J}V_k,\quad
    \ker P_J=\rbsdec_{k\in\bbN\setminus J}V_k\]
  and hence the topological direct sum.
\end{proof}

\begin{remark}
  If $(V_k)_{k\in\bbN}$ is a Riesz basis of finite-dimensional subspaces,
  then we may choose a basis $(v_{k1},\dots,v_{kn_k})$ in each $V_k$.
  The resulting system $(v_{kj})_{k,j}$ is called a 
  \emph{Riesz basis with parentheses}:
  Every $x\in V$ has a unique representation
  \[x=\sum_{k=0}^\infty \Biggl(\sum_{j=1}^{n_k}\alpha_{kj} v_{kj}\Biggr),
    \qquad \alpha_{kj}\in\bbC,\]
  where the series over $k$ converges unconditionally.
\end{remark}

\section{Finitely spectral Riesz bases of subspaces}
\label{sec:invrbs}

We recall some concepts for a linear operator $T$ on a Banach space $V$,
see also \cite{akhiezer-glazman,kato}.
A point $z\in\bbC$ is called a \emph{point of regular type} if
$T-z$ is injective and the inverse $(T-z)^{-1}$ (defined on $\range(T-z)$)
is bounded. The set of all points of regular type is denoted by $r(T)$;
it is open and satisfies $\varrho(T)\subset r(T)$ and 
$\sigma_p(T)\cap r(T)=\varnothing$.

Let $T$ be a closed operator. A subspace $A\subset V$ is called a \emph{core} 
for $T$
if for every $x\in\mdef(T)$ there is a sequence $(x_n)$ in $A$ such that
$\lim x_n=x$ and $\lim Tx_n=Tx$.

Finally we denote by $\rsub(\lambda)$ the space of generalised eigenvectors
or root subspace of $T$
corresponding to the eigenvalue $\lambda\in\sigma_p(T)$, i.e.
\[\rsub(\lambda)=\bigcup_{k\in\bbN}\ker(T-\lambda)^k.\]
For $\lambda\not\in\sigma_p(T)$ we set $\rsub(\lambda)=\{0\}$.
A sequence $x_1,\dots,x_n\in\rsub(\lambda)$ is called a \emph{Jordan chain}
if $(T-\lambda)x_k=x_{k-1}$ for $k\geq 2$ and $(T-\lambda)x_1=0$.

\begin{definition}\label{def:invrbs}
  Let $T$ be a closed operator on a separable Hilbert space $V$.
  We say that a Riesz basis of subspaces 
  $(V_k)_{k\in\bbN}$ of $V$ is 
  \emph{finitely spectral} for $T$ if 
  each $V_k$ is finite-dimensional, $T$-invariant, 
  $V_k\subset\mdef(T)$,
  the sets $\sigma(T|_{V_k})$ are pairwise disjoint,
  and $\sum_{k\in\bbN} V_k$ is a core for $T$.
\end{definition}

\begin{prop}\label{prop:invrbs}
  Let $T$ be a closed operator with a finitely spectral Riesz basis 
  of subspaces $(V_k)_{k\in\bbN}$. Then
  \begin{align}\label{invrbs-domdef}
    \mdef(T)&=\biggl\{x=\sum_{k\in\bbN} x_k\,\bigg|\,x_k\in V_k,\,
      \sum_{k\in\bbN}\|Tx_k\|^2<\infty\biggr\}\,,\\
    Tx&=\sum_{k\in\bbN} Tx_k\quad\text{for}\quad
    x=\sum_{k\in\bbN} x_k\in\mdef(T),\, x_k\in V_k.\label{invrbs-opform}
  \end{align}
  $T$ is bounded if and only if the restrictions $T|_{V_k}$ are
  uniformly bounded and in this case (with $c$ from \eqref{rbs-ineq})
  \[\|T\|\leq c\,\sup_{k\in\bbN}\|T|_{V_k}\|.\]
\end{prop}
\begin{proof}
Let $P_k$ be the projections onto the $V_k$ corresponding to the 
Riesz basis.

(i): We derive (\ref{invrbs-domdef}) and (\ref{invrbs-opform}).
First note that for $u\in\sum_kV_k$ we have $P_kTu=TP_ku$ for all $k$
since $u$ is a finite sum of elements from the $T$-invariant subspaces $V_k$.
Let now $y\in\mdef(T)$. Since $\sum_{k} V_k$ is a core for $T$, there 
is a sequence $y_n\in\sum_{k} V_k$ with $y_n\to y$, $Ty_n\to Ty$.
Since the restriction $T|_{V_k}$ is bounded, we obtain
\[P_kTy=\lim_{n\to\infty}P_kTy_n=\lim_{n\to\infty}T|_{V_k}P_ky_n
  =T|_{V_k}\lim_{n\to\infty}P_ky_n=TP_ky.\]
Hence $\sum_{k}\|TP_ky\|^2=\sum_{k}\|P_kTy\|^2\leq c\|Ty\|^2<\infty$ and
\begin{align*}
  y&=\sum_{k}P_ky\in
  \biggl\{x=\sum_{k}x_k
  \,\bigg|\, x_k\in V_k,\, \sum_{k}\|Tx_k\|^2<\infty\biggr\}\quad\text{with}\\
  Ty&=\sum_{k}P_kTy=\sum_{k}TP_ky.
\end{align*}
If on the other hand $x=\sum_{k}x_k$ with $x_k\in V_k$,
$\sum_{k}\|Tx_k\|^2<\infty$, then
\[\mdef(T)\ni\sum_{k=0}^nx_k\to x\quad\text{and}\quad
  T\sum_{k=0}^nx_k=\sum_{k=0}^nTx_k\to\sum_{k=0}^\infty Tx_k.\]
Hence $x\in\mdef(T)$ since $T$ is closed.

(ii): Suppose that $L=\sup_{k}\|T|_{V_k}\|<\infty$. Then for 
$x=\sum_{k}x_k\in\mdef(T)$:
\[\|Tx\|^2=\big\|\sum_{k}T|_{V_k}x_k\big\|^2\leq
    c\sum_{k}\|T|_{V_k}x_k\|^2
  \leq cL^2\sum_{k}\|x_k\|^2\leq c^2L^2\|x\|^2;\]
thus $T$ is bounded with norm $\leq c\,L$.
\end{proof}

For the case that the $V_k$ are pairwise orthogonal and possibly
infinite-dimen\-sional, 
the spectrum of an operator defined by \eqref{invrbs-domdef},
\eqref{invrbs-opform}
was calculated by Davies \cite[Theorem~8.1.12]{davies}.
We obtain:
\begin{coroll}\label{coroll:invrbs}
  Let $T$ be a closed operator with a finitely spectral Riesz basis 
  of subspaces $(V_k)_{k\in\bbN}$. Then
  \begin{align}
    \sigma_p(T)&=\bigcup_{k\in\bbN}\sigma(T|_{V_k}),\label{invrbs-point}\\
    V_k&=\sum_{\lambda\in\sigma(T|_{V_k})}\rsub(\lambda),\label{invrbs-root}\\
    \varrho(T)=r(T)&=\Bigl\{z\in\bbC\setminus\sigma_p(T)
      \,\Big|\,\sup_{k\in\bbN}\|(T|_{V_k}-z)^{-1}\|<\infty\Bigr\}.
      \label{invrbs-resolv}
  \end{align}
\end{coroll}
\begin{proof}
  For the identities \eqref{invrbs-point} and \eqref{invrbs-root},
  note that if $\lambda\in\sigma_p(T)$ and 
  $x=\sum_{j\in\bbN}x_j\in\rsub(\lambda)\setminus\{0\}$, $x_j\in V_j$, then
  by \eqref{invrbs-opform}
  \[0=(T-\lambda)^nx=\sum_{j\in\bbN}(T|_{V_j}-\lambda)^nx_j\]
  for some $n\in\bbN$, which implies $(T|_{V_j}-\lambda)^nx_j=0$ for all
  $j$. Since $x_{k}\neq 0$ for some $k$, we obtain
  $\lambda\in\sigma(T|_{V_{k}})$. As the $\sigma(T|_{V_j})$ are disjoint,
  we have $\lambda\not\in\sigma(T|_{V_j})$ and hence $x_j=0$ for $j\neq k$,
  i.e.\ $x\in V_{k}$.

  To show \eqref{invrbs-resolv}, first note that if $z\in r(T)$, then
  for every $k\in\bbN$, $(T|_{V_k}-z)^{-1}$ exists and is a restriction
  of $(T-z)^{-1}$, thus 
  \(\sup_k\|(T|_{V_k}-z)^{-1}\|\leq\|(T-z)^{-1}\|<\infty\).
  Furthermore, if $z\in\bbC\setminus\sigma_p(T)$ with
  $\sup_{k}\|(T|_{V_k}-z)^{-1}\|<\infty$, then
  \[S:\sum_{k\in\bbN}x_k\mapsto\sum_{k\in\bbN}(T|_{V_k}-z)^{-1}x_k\]
  defines a bounded operator $S:V\to V$ satisfying
  $(T-z)Sx=x$ for all $x\in V$. Consequently
  $z\in\varrho(T)$ with $(T-z)^{-1}=S$.
\end{proof}

In some situations,
the conditions on the closedness and
the core in Definition~\ref{def:invrbs} are automatically fulfilled:
\begin{prop}\label{prop:invrbs-core}
  Let $T$ be an operator on $V$, $(V_k)_{k\in\bbN}$ a Riesz basis of 
  finite-dimen\-sional,
  $T$-invariant subspaces of\, $V$, $V_k\subset\mdef(T)$ for all $k$, and
  $\sigma(T|_{V_k})$ pairwise disjoint. Then:
  \begin{itemize}
  \item[(i)] $T_0=T|_{\sum_kV_k}$ is closable and $(V_k)_{k\in\bbN}$ is 
    finitely spectral for $\overline{T_0}$.
  \item[(ii)] If\, $r(T)\neq\varnothing$, then $T$ is closable and
    $(V_k)_{k\in\bbN}$ is finitely spectral for $\overline{T}$.
  \end{itemize}
\end{prop}
\begin{proof}
  (i): Let $x_n\in\mdef(T_0)=\sum_kV_k$ with $\lim x_n=0$ and $\lim T_0x_n=y$.
  As in the proof of Proposition~\ref{prop:invrbs} we have
  \[P_ky=\lim_{n\to\infty}P_kT_0x_n=\lim_{n\to\infty}T|_{V_k}P_kx_n
    =T|_{V_k}P_k\lim_{n\to\infty}x_n=0\]
  for every $k\in\bbN$ and hence $y=0$; $T_0$ is closable.
  The other assertion is now immediate.

  (ii): In view of (i) it suffices to show $T\subset\overline{T_0}$;
  for $T$ is closable then, and from $T_0\subset T$ we conclude 
  $\overline{T_0}=\overline{T}$.
  Let $x\in\mdef(T)$ and $z\in r(T)$. Using the Riesz basis $(V_k)$, we 
  have the expansion $(T-z)x=\sum_{k=0}^\infty y_k$ with $y_k\in V_k$.
  Since $T-z$ is injective and $V_k$ is finite-dimensional and
  $T$-invariant, $T-z$ maps $V_k$ onto $V_k$.
  We can thus set $x_k=(T-z)^{-1}y_k\in V_k$ and obtain
  $x=\sum_{k=0}^\infty x_k$ by the boundedness of $(T-z)^{-1}$.
  Consequently 
  \[\mdef(T_0)\ni\sum_{k=0}^nx_k\to x \qquad\text{and}\qquad
    (T_0-z)\sum_{k=0}^nx_k=\sum_{k=0}^ny_k\to(T-z)x\]
  as $n\to\infty$, i.e., $x\in\mdef(\overline{T_0})$ and
  $\overline{T_0}x=Tx$.
\end{proof}

The notion of a finitely spectral Riesz basis of subspaces contains
many other types of bases related to eigenvectors and the spectrum
as special cases:
\begin{prop}\label{prop:invrbs-relat}
  Let $T$ be closed with $r(T)\neq\varnothing$ and $\dim\rsub(\lambda)<\infty$
  for all $\lambda\in\sigma_p(T)$.
  Then for the assertions
  \begin{itemize}
  \item[(i)] $T$ has a finitely spectral Riesz basis of subspaces,
  \item[(ii)] the root subspaces $\rsub(\lambda)$ of\, $T$ form a 
    Riesz basis,
  \item[(iii)] $T$ has a Riesz basis of Jordan chains,
  \end{itemize}
  we have
  \((iii)\Rightarrow(ii)\Rightarrow(i).\)
\end{prop}
\begin{proof}
  (ii)$\Rightarrow$(i) is trivial. For (iii)$\Rightarrow$(ii) consider
  for each eigenvalue $\lambda\in\sigma_p(T)$ the subspace
  $V_\lambda$ generated by all Jordan chains from the basis which
  correspond to $\lambda$.
  Then $(V_\lambda)_{\lambda\in\sigma_p(T)}$ is
  a Riesz basis of subspaces and
  $V_\lambda=\rsub(\lambda)$.
\end{proof}

\begin{remark}\label{rem:invrbs-relat}
  In the situation of the previous proposition,  assertion (i)
  is equivalent to the existence of
  a \emph{Riesz basis with parentheses of Jordan chains} with the 
  additional property that Jordan chains corresponding to the
  same eigenvalue lie inside the same parenthesis.

  If $T$ has a compact resolvent, then (ii) holds if and only if
  $T$ is a \emph{spectral operator} in the sense of 
  Dunford, see \cite{dunford-schwartz3,wyss-phd}.

  A closed operator $T$ is called 
  \emph{Riesz-spectral} \cite{curtain-zwart,kuiper-zwart} if all its
  eigenvalues are simple, $T$ has a Riesz basis of eigenvectors, and 
  $\overline{\sigma_p(T)}$ is totally disconnected. So if $T$ is
  Riesz-spectral then (iii) holds.
\end{remark}

For an operator $G$ let $N(r,G)$ be the sum of the algebraic multiplicities
$\dim\rsub(\lambda)$ for all $\lambda\in\sigma_p(G)$ with $|\lambda|\leq r$.
An operator $S$ is called \emph{$p$-subordinate} to $G$ with $0\leq p\leq 1$
if $\mdef(G)\subset\mdef(S)$ and there exists $b\geq 0$ such that
\[\|Sx\|\leq b\|x\|^{1-p}\|Gx\|^p \quad\text{for}\quad
  x\in\mdef(G).\]
\begin{theo}
  \label{theo:psubpert-invrbs}
  Let $G$ be a normal operator with compact resolvent whose eigenvalues lie on
  a finite number of rays $e^{i\theta_j}\bbR_{\geq0}$, $0\leq\theta_j<2\pi$,
  from the origin. 
  Let $S$ be
  $p$-subordinate  to $G$ with $0\leq p<1$.
  If
  \[\liminf_{r\to\infty}\frac{N(r,G)}{r^{1-p}}<\infty,\]
  then $T=G+S$ has a compact resolvent and a finitely spectral
  Riesz basis of subspaces $(V_k)_{k\in\bbN}$.
\end{theo}
\begin{proof}
  See Theorems~4.5 and~6.1 in \cite{wyss-psubpert}.
  In particular, note that the $V_k$ were constructed as the ranges of
  Riesz projections associated with disjoint parts of $\sigma(T)$, and hence
  the $\sigma(T|_{V_k})$ are disjoint.
\end{proof}

Now we study invariant subspaces with respect to a finitely spectral
Riesz basis of subspaces.
\begin{lemma}\label{lem:invcomp}
  Let $T$ be a closed operator with a finitely spectral Riesz basis
  of subspaces $(V_k)_{k\in\bbN}$. For a compatible subspace
  $U=\bigoplus^2_{k\in\bbN}U_k$, $U_k\subset V_k$, the following
  assertions are equivalent:
  \begin{itemize}
  \item[(i)] $U$ is $T$-invariant;
  \item[(ii)] all $U_k$ are $T$-invariant.
  \end{itemize}
  For $z\in\varrho(T)$, (i) and (ii) are equivalent to
  \begin{itemize}
  \item[(iii)] $U$ is $(T-z)^{-1}$-invariant.
  \end{itemize}
\end{lemma}
\begin{proof}
  The claim is immediate from Proposition~\ref{prop:invrbs},
  in particular \eqref{invrbs-opform}.
  For $z\in\varrho(T)$ note that $\dim U_k<\infty$ and $U_k\subset\mdef(T)$
  imply that 
  $U_k$ is $T$-invariant if and only if $U_k$ is $(T-z)^{-1}$-invariant.
\end{proof}

\begin{coroll}\label{coroll:invcomp}
  The subspace $U$ is $T$-invariant and compatible with $(V_k)_{k\in\bbN}$
  if and only if 
  \begin{equation}\label{eq:invcomp-var}
    U=\overline{\sum_{\lambda\in\sigma_p(T)}W_\lambda}
  \end{equation}
  with $T$-invariant subspaces $W_\lambda\subset\rsub(\lambda)$.
  In particular, for $\sigma\subset\sigma_p(T)$ we obtain the compatible subspace
  \begin{equation}\label{eq:assoccomp}
    U_\sigma=\overline{\sum_{\lambda\in\sigma}\rsub(\lambda)}
  \end{equation}
  associated with $\sigma$.
\end{coroll}
\begin{proof}
  If $U=\bigoplus^2_kU_k$ with $U_k\subset V_k$ $T$-invariant, then,
  since $\dim U_k<\infty$,
  \begin{equation}\label{eq:invcomp-fine}
    U_k=\sum_{\lambda\in\sigma(T|_{V_k})}W_\lambda
  \end{equation}
  with $W_\lambda\subset\rsub(\lambda)$ $T$-invariant; consequently
  \eqref{eq:invcomp-var}. 
  On the other hand, if $U$ is given by \eqref{eq:invcomp-var}, and we define
  $U_k$ by \eqref{eq:invcomp-fine}, then $U_k$ is $T$-invariant,
  $U_k\subset V_k$, and we obtain $U=\bigoplus^2_kU_k$.
\end{proof}

In the following, we will use the  notation $\sigma_p^i(T)$,
$\sigma_p^+(T)$ and $\sigma_p^-(T)$ for the set of eigenvalues of $T$
on the imaginary axis and in the open right and left half-plane, respectively.
\begin{remark}\label{rem:dichot}
  Let $T$ be a closed operator with a finitely spectral Riesz basis
  of subspaces $(V_k)_{k\in\bbN}$,
  $\sigma_p^i(T)=\varnothing$, and consider the invariant compatible
  subspaces $U_\pm$
  associated with $\sigma_p^\pm(T)$. We have $U_\pm=\bigoplus^2_kV_k^\pm$
  where $V_k=V_k^+\oplus V_k^-$
  and $V_k^\pm$ are the spectral subspaces of $T|_{V_k}$ corresponding to
  the right and left half-plane.
  Hence $U_+\dotplus U_-\subset V$ algebraic direct and dense by
  Proposition~\ref{prop:rbs-decomp}.
  For the operator $T$ in Example~\ref{ex:unbndsol},
  the sum is in fact not topological direct; in particular
  $U_+\dotplus U_-\subsetneqq V$.
  
  On the other hand, if an operator $T$ is \emph{dichotomous}
  (see \cite{langer-ran-rotten}),  then 
  a strip around the imaginary axis belongs to $\varrho(T)$, and
  there is a topological direct decomposition $V=V_+\oplus V_-$ 
  such that $V_\pm$ is $T$-invariant and $\sigma(T|_{V_\pm})$ is contained
  in the right and left half-plane, respectively.
  In particular $U_\pm\subset V_\pm$. Consequently the operator in
  Example~\ref{ex:unbndsol} is not dichotomous.

\end{remark}

\begin{lemma}
  Let $T$ be an operator on $V$, $z_0\in\varrho(T)$ and $U\subset V$ a
  closed $(T-z_0)^{-1}$-invariant subspace. Then $U$ is $(T-z)^{-1}$-invariant
  for all $z$ in the connected component of $z_0$ in $\varrho(T)$.
\end{lemma}
\begin{proof}
  It suffices to show that the set
  \[A=\{z\in\varrho(T)\,|\,U \text{ is $(T-z)^{-1}$-invariant}\}\]
  is relatively open and closed in $\varrho(T)$.
  Let $z\in A$. For small $|w-z|$ a Neumann series argument shows that
  \begin{align*}
    (T-w)^{-1}&=(T-z)^{-1}\bigl(I-(w-z)(T-z)^{-1}\bigr)^{-1}
    =\sum_{k=0}^\infty (w-z)^k(T-z)^{-k-1}
  \end{align*}
  If $x\in U$, then $(T-z)^{-k-1}x\in U$ for all $k\geq 0$. Hence also
  $(T-w)^{-1}x\in U$, i.e.\ $w\in A$; $A$ is an open set.

  Now let $w\in\varrho(T)$ with $w=\lim_{n\to\infty} z_n$, $z_n\in A$.
  For $x\in U$ we then have
  \[U\ni (T-z_n)^{-1}x\to(T-w)^{-1}x\in U
    \quad\text{as}\quad n\to\infty\]
  since the resolvent $(T-z)^{-1}$ is continuous in $z$.
  Hence $w\in A$, i.e., $A$ is relatively closed.
\end{proof}

\begin{prop}\label{prop:invcomp}
  Let $T$ be an operator with compact resolvent and a finitely spectral
  Riesz basis of subspaces $(V_k)_{k\in\bbN}$. If\, $U$ is a closed subspace
  which is $(T-z)^{-1}$-invariant for some $z\in\varrho(T)$, then
  $U$ is $T$-invariant and compatible with $(V_k)_{k\in\bbN}$.
\end{prop}
\begin{proof}
  Since $T$ has a compact resolvent,  $\sigma(T)$ consists
  of isolated eigenvalues only and  $\varrho(T)$ is
  connected. The previous lemma thus implies that $U$ is 
  $(T-z)^{-1}$-invariant for all $z\in\varrho(T)$.
  Let $P_k$ be the projections corresponding to the Riesz basis.
  Since $\sigma_k=\sigma(T|_{V_k})$ is an isolated part of the spectrum,
  $P_k$ is the Riesz projection associated with $\sigma_k$, i.e.
  \begin{equation}\label{rieszproj}
    P_k=\frac{i}{2\pi}\int_{\Gamma_k}(T-z)^{-1}dz
  \end{equation}
  where $\Gamma_k$ is a simply closed, positively oriented integration
  contour with $\sigma_k$ in its interior and $\sigma(T)\setminus\sigma_k$
  in its exterior, see e.g.\ \cite[Theorem~III.6.17]{kato}.
  Consequently $P_k(U)\subset U$, and $U$ is thus
  compatible with $(V_k)$. 
  $T$-invariance is now a consequence of Lemma~\ref{lem:invcomp}.
\end{proof}

\section{Hamiltonian operator matrices}
\label{sec:ham}

We use the following definition of a Hamiltonian operator matrix,
see also \cite{azizov-dijksma-gridneva}.

\begin{definition}\label{def:hamopmat}
  Let $H$ be a Hilbert space. A 
  \emph{Hamiltonian operator matrix}
  is a block operator matrix
  \begin{equation*}
    T=\pmat{A&B\\C&-A^*}, \qquad
    \mdef(T)=(\mdef(A)\cap\mdef(C))\times(\mdef(A^*)\cap\mdef(B))
  \end{equation*}
  acting on $H\times H$ with densely defined linear operators 
  $A$, $B$, $C$ on $H$ such that
  $B$ and $C$ are symmetric and $T$ is densely defined.

  If $B$ and $C$ are both nonnegative (positive, uniformly positive),
  then $T$ is called a \emph{nonnegative} (\emph{positive, uniformly positive,}
  respectively) Hamiltonian operator matrix.\footnote{
    Note that the sign convention 
    $T=\left(\begin{smallmatrix}A&-B\\-C&-A^*\end{smallmatrix}\right)$,
    in particular with nonnegative $B,C$, is also used
    in the literature, e.g.\ in \cite{kuiper-zwart,langer-ran-rotten}.
  }
\end{definition}

Hamiltonian operator matrices are connected to two indefinite inner
products on $H\times H$.
We recall some corresponding notions,
see \cite{azizov-iokhvidov,bognar} for more details:
A vector space $V$ together with an inner product
$\braket{\cdot}{\cdot}$ is called a \emph{Krein space}
if $V$ is also a Hilbert space with scalar product $(\cdot|\cdot)$
and there is a selfadjoint involution
$J:V\to V$ such that 
$\braket{x}{y}=(Jx|y)$ for all $x,y\in V$.

A subspace $U\subset V$ is called \emph{neutral} if 
$\braket{x}{x}=0$ for all $x\in U$.
The \emph{orthogonal complement} of $U$ is defined by
\[U^\Jorth=\{x\in V\,|\,\braket{x}{y}=0 \text{ for all } y\in U\}.\]
Two subspaces $U,W\subset V$ are said to be \emph{orthogonal},
$U\Jorth W$, if $W\subset U^\Jorth$.
$U$ is neutral if and only if $U\subset U^\Jorth$.
The subspace $U$ is called
\emph{non-degenerate} if $U\cap U^\Jorth=\{0\}$.

Let $T$ be a densely defined operator on $V$. It is called 
\emph{symmetric} if $\braket{Tx}{y}=\braket{x}{Ty}$ for all $x,y\in \mdef(T)$.
The \emph{adjoint} of $T$ is defined as the maximal operator $T^\Jadj$ such
that
\[\braket{Tx}{y}=\braket{x}{T^\Jadj y} \quad\text{for all}\quad
  x\in\mdef(T),\,y\in\mdef(T^\Jadj).\]
$T$ is called \emph{selfadjoint} if $T=T^\Jadj$, and in this case
its spectrum $\sigma(T)$ is symmetric with respect to the real axis.

Consider the Krein space inner products on $H\times H$ given by
\begin{equation*}
  \braket{x}{y}=(J_1x|y)\quad\text{with}\quad J_1=\pmat{0&-iI\\iI&0}
\end{equation*}
and
\begin{equation*}
  [x|y]=(J_2x|y) \quad\text{with}\quad J_2=\pmat{0&I\\I&0}.
\end{equation*}
Here $(\cdot|\cdot)$ denotes the usual scalar product on $H\times H$.
The straightforward computation
\begin{align*}
  \Braket{\pmat{A&B\\C&-A^*}\pmat{u\\v}}{\pmat{\tilde{u}\\\tilde{v}}}
  &=i(Au+Bv|\tilde{v})-i(Cu-A^*v|\tilde{u})\\
  &=i(u|A^*\tilde{v}-C\tilde{u})-i(v|-B\tilde{v}-A\tilde{u})\\
  &=\Braket{\pmat{u\\v}}{\pmat{-A&-B\\-C&A^*}\pmat{\tilde{u}\\\tilde{v}}}
\end{align*}
shows that $T$ is \emph{$J_1$-skew-symmetric},  i.e.
\[\braket{Tx}{y}=-\braket{x}{Ty} \quad\text{for all}\quad
  x,y\in \mdef(T).\]
As a consequence, $T$ is always closable.
In the following, additional assumptions on $T$ such as in 
Theorem~\ref{theo:ham-invrbs} or the r0-diagonally dominance in
Section~\ref{sec:bndsol} will often imply that $T$ is already closed.
From
\begin{align*}
  \Real\Bigl[\pmat{A&B\\C&-A^*}\pmat{u\\v}\Big|\pmat{u\\v}\Bigr]
  &=\Real\bigl((Au+Bv|v)+(Cu-A^*v|u)\bigr)\\
  &=(Bv|v)+(Cu|u)
\end{align*}
we obtain that $T$ is nonnegative if and only if it is
\emph{$J_2$-accretive}, i.e.
$\Real[Tx|x]\geq 0$ for all $x\in\mdef(T)$.

Recall that we denote by $\sigma_p^i(T)$,
$\sigma_p^+(T)$ and $\sigma_p^-(T)$ the set of eigenvalues of $T$
on the imaginary axis and in the open right and left half-plane, respectively.
As a consequence of the $J_1$-skew-symmetry of $T$ we obtain:
\begin{prop}\label{prop:ham-specsym}
  Let $T$ be a Hamiltonian operator matrix.
  \begin{itemize}
  \item[(i)] If\, $\lambda,\mu\in\sigma_p(T)$ with 
    $\lambda\neq-\overline{\mu}$, then the root subspaces $\rsub(\lambda)$
    and $\rsub(\mu)$ are $J_1$-orthogonal. 
    In particular $\rsub(\lambda)$ is $J_1$-neutral for 
    $\lambda\not\in\sigma_p^i(T)$.
  \item[(ii)] If\, $T$ has a complete system of root subspaces,
    then $\sigma_p(T)$ is symmetric with respect to the imaginary axis,
    and $\rsub(\lambda)+\rsub(-\overline{\lambda})$ is $J_1$-non-degenerate
    with $\dim\rsub(\lambda)=\dim\rsub(-\overline{\lambda})$
    for every $\lambda\in\sigma_p(T)$.
  \item[(iii)] If there exists $z$ such that $z,-\bar{z}\in\varrho(T)$,
    then $T$ is $J_1$-skew-selfadjoint, i.e.\ $T=-T^\Jadj$, and $\sigma(T)$ is
    symmetric with respect to the imaginary axis.
  \end{itemize}
\end{prop}

In particular, the point spectrum of a Hamiltonian with a finitely spectral
Riesz basis of subspaces is symmetric with respect to the imaginary axis.

\begin{proof}[Proof of the proposition]
  (i): Since $iT$ is $J_1$-symmetric, this is an immediate consequence
  of \cite[Theorem II.3.3]{bognar}.

  (ii): Let
  \[\sigma_0=\sigma_p^i(T)\cup\sigma_p^+(T)\cup
    \set{-\overline{\lambda}}{\lambda\in\sigma_p^-(T)}\]
  and define $U_\lambda=\rsub(\lambda)+\rsub(-\overline{\lambda})$ for
  $\lambda\in\sigma_0$. From (i) it follows that the $U_\lambda$ are
  pairwise $J_1$-orthogonal. For $x\in U_\lambda\cap U_\lambda^\Jorth$
  this implies that $\braket{x}{y}=0$ for all $y\in\sum_\mu U_\mu$.
  Since $\sum_\mu U_\mu\subset H\times H$
  is dense by assumption, we obtain $\braket{x}{y}=0$ for all $y\in H\times H$
  and thus $x=0$; $U_\lambda$ is $J_1$-non-degenerate.
  For $\lambda\in\sigma_0$ with $\Real\lambda>0$, the subspaces 
  $\rsub(\lambda)$ and $\rsub(-\overline{\lambda})$ are neutral and their
  sum is non-degenerate. 
  This implies that $\dim\rsub(\lambda)=\dim\rsub(-\overline{\lambda})$,
  see \cite[\S I.10]{bognar}.
  In particular $\lambda,-\overline{\lambda}\in\sigma_p(T)$
  and hence the symmetry of $\sigma_p(T)$.

  (iii): We have that $iT$ is $J_1$-symmetric and 
  $w,\overline{w}\in\varrho(iT)$
  where $w=iz$.   As in the Hilbert space situation this implies that
  $iT$ is $J_1$-selfadjoint. Consequently, $T$ is $J_1$-skew-selfadjoint.
\end{proof}

The $J_2$-accretivity of a nonnegative Hamiltonian leads to characterisations
of the spectrum at the imaginary axis:
\begin{prop}\label{prop:hamspecgap}
  Let $T$ be a nonnegative Hamiltonian operator matrix.
  \begin{itemize}
  \item[(i)] We have $\sigma^i_p(T)=\varnothing$ if and only if
    \begin{equation}\label{prop:hamspecgap-cond}
      \ker(A-it)\cap\ker C=\ker(A^*+it)\cap\ker B=\{0\}\quad\text{for all}
      \quad t\in\bbR.
    \end{equation}
  \item[(ii)] If\, $T$ is uniformly positive with
    $B,C\geq\gamma$, then
    \[\bigl\{z\in\bbC\,\big|\,|\Real z|<\gamma\bigr\}\subset r(T).\]
  \end{itemize}
\end{prop}

\begin{proof}
  (i): We show that $(T-it)x=0$ for $x=(u,v)\in\mdef(T)$
  if and only if
  \[u\in\ker(A-it)\cap\ker C\quad\text{and}\quad
    v\in\ker(A^*+it)\cap\ker B.\]
  Indeed if $(T-it)x=0$, then
  \begin{gather*}
    (A-it)u+Bv=0,\quad Cu-(A^*+it)v=0 \quad\text{and}\quad\\
     0=\Real(it[x|x])=\Real[Tx|x]=(Bv|v)+(Cu|u).
  \end{gather*}
  Since $B,C$ are nonnegative, this yields $(Bv|v)=(Cu|u)=0$.
  Now $B$ admits a non\-negative selfadjoint extension $\widetilde{B}$.
  We obtain $\|\widetilde{B}^{1/2}v\|^2=(\widetilde{B}v|v)=(Bv|v)=0$ 
  and hence
  $Bv=(\widetilde{B}^{1/2})^2v=0$.
  Similarly $Cu=0$ and thus also $(A-it)u=(A^*+it)v=0$.
  The other implication is immediate.

  (ii): For $x=(u,v)\in\mdef(T)$ we have
  $\Real[Tx|x]=(Bv|v)+(Cu|u)\geq\gamma\|x\|^2$.
  Let $z\in\bbC\setminus r(T)$. Then there exists a sequence
  $x_n\in\mdef(T)$ with $\|x_n\|=1$ and $(T-z)x_n\to 0$ as $n\to\infty$.
  For $\alpha_n=\Real[(T-z)x_n|x_n]$ this implies 
  $\alpha_n\to 0$.
  We obtain 
  \begin{align*}
    \gamma&=\gamma\|x_n\|^2\leq\Real[Tx_n|x_n]
    =\alpha_n+\Real z\cdot[x_n|x_n]\\
    &\leq |\alpha_n|+|\Real z|\,|(J_2x_n|x_n)|\leq
    |\alpha_n|+|\Real z|\|x_n\|^2\to|\Real z|
  \end{align*}
  as $n\to\infty$, i.e.\ $\gamma\leq|\Real z|$.
\end{proof}

We end this section with two perturbation theorems which ensure the
existence of finitely spectral Riesz bases of subspaces for $T$.
\begin{theo}\label{theo:ham-invrbs}
  Let $T$ be a Hamiltonian operator matrix where $A$ is normal with
  compact resolvent and $B$, $C$ are \emph{$p$-subordinate} to
  $A$ with $0\leq p<1$.
  If $\sigma(A)$ lies on finitely many rays
  from the origin and
  \begin{equation}\label{theo:ham-invrbs-ac}
    \liminf_{r\to\infty}\frac{N(r,A)}{r^{1-p}}<\infty,
  \end{equation}
  then $T$ has a compact resolvent, is $J_1$-skew-selfadjoint, and there
  exists a finitely spectral Riesz basis of subspaces $(V_k)_{k\in\bbN}$
  for $T$.
\end{theo}

\begin{proof}
  This is an application of Theorem~\ref{theo:psubpert-invrbs}
  to the decomposition
  \[T=G+S\quad\text{with}\quad G=\pmat{A&0\\0&-A^*},\quad 
    S=\pmat{0&B\\C&0},\]
  see \cite[Theorem~7.2]{wyss-psubpert} for details. The skew-selfadjointness
  then follows by Proposition~\ref{prop:ham-specsym}.
\end{proof}

\begin{theo}\label{theo:ham-rbeigvec}
  Let $T$ be a uniformly positive Hamiltonian such that $A$ is 
  skew-selfadjoint with compact resolvent,  $B,C$ are bounded
  and satisfy $B,C\geq\gamma$.
  Let $ir_k$ be the eigenvalues of $A$
  where $(r_k)_{k\in\Lambda}$ is increasing and
  $\Lambda\in\{\bbZ_+,\bbZ_-,\bbZ\}$.
  Suppose that almost all eigenvalues $ir_k$ are simple  and that for some
  $l>b=\max\{\|B\|,\|C\|\}$ we have
  \[r_{k+1}-r_k\geq 2l \quad\text{for almost all } k\in\Lambda.\]
  Then $T$ has a compact resolvent, almost all of its eigenvalues are simple,
  \[\sigma(T)\subset\{z\in\bbC\,|\,\gamma\leq|\Real z|\leq b\},\]
  and $T$ admits a Riesz basis of eigenvectors and finitely many Jordan 
  chains.
\end{theo}
\begin{proof}
  See \cite[Theorem~7.3]{wyss-psubpert}.
\end{proof}

\begin{remark}
  Due to \cite[Remark~6.7]{wyss-psubpert}, Theorem~\ref{theo:ham-invrbs}
  continues
  to hold if $A$ is an operator with compact resolvent and a Riesz basis
  of Jordan chains,
  $B$ is $p$-subordinate to $A^*$, $C$ is $p$-subordinate to $A$, 
  $0\leq p<1$, almost all eigenvalues of $A$ lie inside sets
  \(\{e^{i\theta_j}(x+iy)\,|\,x>0,\,|y|\leq\alpha x^p\}\)
  with $\alpha\geq0$, $-\pi\leq\theta_j<\pi$, $j=1,\dots,n$,
  and \eqref{theo:ham-invrbs-ac} is satisfied.
  Theorem~\ref{theo:ham-rbeigvec} also holds if $A$ has a compact resolvent, 
  a Riesz 
  basis of eigenvectors and finitely many Jordan chains, and almost all
  eigenvalues of $A$ are simple and contained in a strip around the imaginary
  axis; the constant $b$ has to be adjusted then.
\end{remark}

\section{Invariant subspaces of Hamiltonians}
\label{sec:invsubham}

Now we investigate properties of certain invariant subspaces of the 
Hamiltonian with respect to the two indefinite inner products 
defined in the previous section.

Let $V$ be a Krein space.
Recall that a subspace $U\subset V$ 
is neutral if and only if $U\subset U^\Jorth$. It is called 
\emph{hypermaximal neutral} if $U=U^\Jorth$, see
\cite{azizov-iokhvidov,bognar}.
It is not hard to see that if $U,W$ are neutral subspaces with $V=U\oplus W$,
then $U$ and $W$ are hypermaximal neutral.
For $\dim V<\infty$, this is even an equivalence:

\begin{lemma}
  Let $V$ be a finite-dimensional Krein space.
  If\, $U\subset V$ is hypermaximal
    neutral, then there exists a neutral subspace $W$ such that
    $V=U\oplus W$.
\end{lemma}
\begin{proof}
  By induction on $n=\dim U$ we show that  there exist systems
  $(e_1,\dots,e_n)$ in $U$ and $(f_1,\dots,f_n)$ in $V$ which
  form a \emph{dual pair},
  i.e.\ $\braket{e_j}{f_l}=\delta_{jl}$, and are such that
  $W=\mspan\{f_1,\dots,f_n\}$ is neutral.
  Indeed, if $\dim U=n+1$ and $e\in U\setminus\mspan\{e_1,\dots,e_n\}$,
  we can set
  \[e_{n+1}=e-\sum_{j=1}^n\braket{e}{f_j}e_j.\]
  Since $V$ is non-degenerate, there exists $f\in V$ with 
  $\braket{e_{n+1}}{f}=1$. Then
  \[\widetilde{f}=f-\sum_{j=1}^n\braket{f}{e_j}f_j
    -\sum_{j=1}^n\braket{f}{f_j}e_j
    \quad\text{and}\quad
    f_{n+1}=\widetilde{f}-\frac{\braket{\widetilde{f}}{\widetilde{f}}}{2}
    e_{n+1}\]
  yields the desired properties.

  If  $\sum_{j=1}^n\alpha_je_j+\beta_jf_j=0$, then
  we can take the inner product of this equation with the elements $e_j,f_j$
  and find $\alpha_j=\beta_j=0$ for all $j$;
  $(e_1,\dots,e_n,f_1,\dots,f_n)$ is linearly independent.
  In particular $(e_1,\dots,e_n)$ is a basis of $U$ and  $U\cap W=\{0\}$.
  To show  $V=U\oplus W$, let $x\in V$ and set $u=x-w$ where
  $w=\sum_{j=1}^n\braket{x}{e_j}f_j\in W$. 
  Then $\braket{u}{e_j}=0$ for all $j$, i.e.\ $u\in U^\Jorth=U$.
\end{proof}

For an operator whose point spectrum $\sigma_p(T)$ is symmetric with
respect to the imaginary axis, we say that
a subset $\sigma\subset\sigma_p(T)\setminus i\bbR$ is an 
\emph{sc-set} (\emph{sc} for skew-conjugate) if
\begin{itemize}
\item[(i)]
  \(\lambda\in\sigma\,\Rightarrow\,-\overline{\lambda}\not\in\sigma\)
  and
\item[(ii)]
  $\lambda\in\sigma_p(T)\setminus i\bbR\,\Rightarrow\,\lambda\in\sigma$
  or $-\overline{\lambda}\in\sigma$.
\end{itemize}
In other words, $\sigma$ contains one eigenvalue from each
skew-conjugate pair $(\lambda,-\overline{\lambda})$ in
$\sigma_p(T)\setminus i\bbR$.

\begin{theo}\label{theo:hypmaxinv}
  Let $T$ be a closed Hamiltonian operator matrix with a finitely spectral
  Riesz basis of subspaces.
  Then $T$ admits a hypermaximal $J_1$-neutral, $T$-invari\-ant, compatible
  subspace if and only if   for all $it\in\sigma_p^i(T)$
  we have
  \begin{equation}\label{eq:hypmaxcond}
    \rsub(it)=M_{it}\oplus N_{it}
    \quad\text{with}\quad 
    M_{it},N_{it} \text{ $J_1$-neutral and } M_{it}
    \text{ $T$-invariant}.
  \end{equation}
  In this case, for every sc-set $\sigma\subset\sigma_p(T)\setminus i\bbR$
  the  $T$-invariant compatible subspace
  \begin{equation}\label{eq:hypmaxinv}
    U=\overline{\sum_{\lambda\in\sigma}\rsub(\lambda)
      +\sum_{it\in\sigma_p^i(T)}M_{it}}
  \end{equation}
  is hypermaximal $J_1$-neutral.
\end{theo}

\begin{proof}
  Let $(V_k)_{k\in\bbN}$ be a finitely spectral Riesz basis of subspaces
  for $T$ and write $\sigma_k=\sigma(T|_{V_k})$.
  Suppose first that $U$ is hypermaximal $J_1$-neutral, $T$-invariant,
  and compatible with $(V_k)$. So $U$ is of the form
  \[U=\rbsdec_{k\in\bbN} U_k
    =\overline{\sum_{\lambda\in\sigma_p(T)}M_\lambda}\]
  where the subspaces $U_k\subset V_k$ and $M_\lambda\subset\rsub(\lambda)$
  are all $T$-invariant, compare Corollary~\ref{coroll:invcomp}.
  By Proposition~\ref{prop:ham-specsym}, each $\rsub(it)$, 
  $it\in\sigma_p^i(T)$, is 
  $J_1$-non-degenerate and thus itself a Krein space.
  In view of the previous lemma it suffices to show that
  $M_{it}$ is hypermaximal neutral with respect to $\rsub(it)$, i.e., 
  $M_{it}^\Jorth\cap\rsub(it)=M_{it}$.

  Since $M_{it}\subset U$ we have that $M_{it}$ is neutral and hence
  $M_{it}\subset M_{it}^\Jorth\cap\rsub(it)$.
  Let $x\in M_{it}^\Jorth\cap\rsub(it)$.
  Since $\rsub(it)$ is $J_1$-orthogonal to $\rsub(\lambda)$ for every
  $\lambda\neq it$, we see that $x\Jorth M_\lambda$ for all $\lambda$
  and hence $x\in U^\Jorth=U$.
  On the other hand $x\in\rsub(it)\subset V_{k_0}$ with $k_0$ such that
  $it\in\sigma_{k_0}$. Consequently $x\in U\cap V_{k_0}=U_{k_0}$.
  Now the decomposition
  \[U_{k_0}=\bigoplus_{\lambda\in\sigma_{k_0}}M_\lambda\]
  implies that $x\in U_{k_0}\cap\rsub(it)=M_{it}$.

  For the other implication, suppose now that
  for every $it\in\sigma_p^i(T)$ there is a decomposition
  $\rsub(it)=M_{it}\oplus N_{it}$ into neutral subspaces
  where $M_{it}$ is $T$-invariant, let
  $\sigma\subset\sigma_p(T)\setminus i\bbR$ be an sc-set, and let $U$ be
  given by \eqref{eq:hypmaxinv}.
  Since $U$ is the closure of the sum of neutral, pairwise orthogonal
  subspaces, $U$ is neutral. Moreover, $U$ is $T$-invariant and
  compatible with $(V_k)$ with decomposition
  \[U=\rbsdec_{k\in\bbN}U_k,\qquad
    U_k=\sum_{\lambda\in\sigma_k\cap\sigma}\rsub(\lambda)
    +\sum_{it\in\sigma_k^i}M_{it},\]
  where $\sigma_k^i=\sigma_p^i(T|_{V_k})$.
  It remains to show that $U^\Jorth\subset U$.
  We have $V_k=U_k\oplus W_k$ with
  \[W_k=\sum_{\lambda\in\tau_k}\rsub(\lambda)+\sum_{it\in\sigma_k^i}N_{it},
    \qquad \tau_k=\sigma_k\setminus(\sigma\cup\sigma_k^i).\]
  Let $x\in U^\Jorth$. We expand $x$ in the Riesz basis $(V_k)$ as
  $x=\sum_k(u_k+w_k)$ with $u_k\in U_k$, $w_k\in W_k$.
  To show that all $w_k$ are zero, we consider now the subspaces
  \[\widetilde{U}_k=\sum_{\lambda\in\tau_k}\rsub(-\overline{\lambda})
    +\sum_{it\in\sigma_k^i}M_{it}.\]
  The fact that $\sigma$ is an sc-set yields 
  $\lambda\in\tau_k\Rightarrow -\overline{\lambda}\in\sigma$,
  and therefore  $\widetilde{U}_k\subset U$.
  Moreover $\widetilde{U}_k$ is $J_1$-orthogonal to $W_j$ for $j\neq k$,
  and $W_k$ is neutral.
  For $\tilde u\in \widetilde{U}_k$, $\tilde w\in W_k$
  we thus compute
  \[0=\braket{x}{\tilde u}=\sum_{j\in\bbN}\braket{u_j+w_j}{\tilde u}
    =\braket{w_k}{\tilde u}=\braket{w_k}{\tilde u+\tilde w}.\]
  In view of Proposition~\ref{prop:ham-specsym}, $\widetilde{U}_k+W_k$
  is non-degenerate since it is the orthogonal sum of subspaces 
  $\rsub(\lambda)+\rsub(-\overline{\lambda})$, 
  $\lambda\in\tau_k\cup\sigma_k^i$.
  Consequently $w_k=0$ for all $k$ and hence $x=\sum_ku_k\in U$.
\end{proof}

\begin{remark}
  Since all root subspaces of $T$ are finite-dimensional,
  results  about the Jordan structure of $J$-symmetric matrices
  (e.g.\ \cite[Theorem~2.3.2]{lancaster-rodman})
  may be used to reformulate condition \eqref{eq:hypmaxcond}:
  It turns out that
  \eqref{eq:hypmaxcond} holds if and only if
  $\rsub(it)=M_{it}'\oplus N_{it}'$ with neutral subspaces $M_{it}',N_{it}'$.
\end{remark}

Now we consider the subspaces associated with $\sigma_p^\pm(T)$, the
point spectrum of $T$ in the right and left half-plane, respectively.  
\begin{lemma}\label{lem:resolvint}
  Let $T$ be an operator on a Banach space with
  $\sigma_p^i(T)=\varnothing$.
  Consider the algebraic direct decomposition
  \[\sum_{\lambda\in\sigma_p(T)}\rsub(\lambda)
    =W_+\dotplus W_-, \qquad 
    W_\pm=\sum_{\lambda\in\sigma_p^\pm(T)}\rsub(\lambda),\]
  and the associated algebraic projections $P_\pm$ onto $W_\pm$.
  Then
  \begin{equation}\label{resolvint}
    \frac{1}{i\pi}\int_{i\bbR}^\prime(T-z)^{-1}x\,dz=P_+x-P_-x\qquad
    \text{for all}\quad x\in\!\sum_{\lambda\in\sigma_p(T)}\!\rsub(\lambda)\,,
  \end{equation}
  where the prime denotes the Cauchy principal value at infinity, that is
  $\int^\prime_{i\bbR}f\,dz=\lim_{r\to\infty}\int_{-ir}^{ir}f\,dz$.
\end{lemma}

Note that the integrand in \eqref{resolvint} is well-defined
since $(T-z)^{-1}$ acts, for each $x$, on a
finite sum of finite-dimensional subspaces generated by Jordan chains;
$(T-z)^{-1}x$ is thus continuous in $z$.

\begin{proof}[Proof of the lemma]
  By linearity it suffices to consider
  $x\in \rsub(\lambda)$ and the Jordan chain generated by $x$. 
  With respect to this Jordan chain, $T$ is represented 
  by the matrix
  \begin{equation}\label{jordanblock}
    E_\lambda=\begin{pmatrix}\lambda&1&\\&\sddots&\sddots\\&&\lambda
    \end{pmatrix},
  \end{equation}
  and it suffices to show that
  \[\int_{i\bbR}^\prime(E_\lambda-z)^{-1}dz=\pm i\pi I\]
  for $\Real\lambda\gtrless0$.
  This is a straightforward calculation.
\end{proof}

\begin{lemma}\label{lem:resolvint-estim}
  Let $T$ be an operator with a 
  Riesz basis $(x_k)_{k\in\bbN}$ consisting of Jordan chains.
  If\, $\sigma_p^i(T)=\varnothing$ and
  $\sigma_p(T)$ is contained
  in a strip around the imaginary axis, then
  \begin{equation*}
    \int_{-\infty}^\infty\|(T-it)^{-1}x\|^2\,dt\,\geq\, c\|x\|^2
    \qquad\text{for}\quad x\in\mspan\{x_k\,|\,k\in\bbN\}
  \end{equation*}
  with some constant $c>0$.
\end{lemma}

\begin{proof}
  Let $x\in\mspan\{x_k\,|\,k\in\bbN\}$. Then there is a finite 
  system $F=(y_1,\ldots,y_n)\subset (x_k)_{k\in\bbN}$
  consisting of Jordan chains
  such that $x=\alpha_1y_1+\ldots+\alpha_ny_n$. 
  $\mspan F$ is a $T$-invariant subspace with basis $F$.
  With respect to $F$, $(T-it)^{-1}$ is represented by a block diagonal 
  matrix $D$ with blocks of the form $(E_\lambda-it)^{-1}$, $E_\lambda$ as in
  \eqref{jordanblock}.
  Hence
  \[(T-it)^{-1}x=\sum_{k=1}^n\alpha_k(T-it)^{-1}y_k=\sum_{j,k=1}^n\alpha_k
    D_{jk}y_j.\]
  Let $m,M>0$ be the constants from \eqref{eq:rb-ineq} for the Riesz basis
  $(x_k)$.
  Putting $\xi=(\alpha_1,\ldots,\alpha_n)$ and using the Euclidean norm on
  $\bbC^n$, we find
  \[\|(T-it)^{-1}x\|^2\geq m\sum_{j=1}^n\Big|\sum_{k=1}^n\alpha_kD_{jk}\Big|^2
    =m\|D\xi\|^2.\]
  Now $\|D\xi\|^2$ is the sum of terms of the form
  $\|(E_\lambda-it)^{-1}\nu\|^2$, one for each Jordan chain in $F$ with
  $\nu$ the
  part of $\xi$ corresponding to that Jordan chain. From
  \[\|E_\lambda-it\|\leq|\lambda-it|+\big\|\left(\begin{smallmatrix}0&1&\\
    &\sddots&\sddots\\&&0\end{smallmatrix}\right)\big\|\leq|\lambda-it|+1\]
  it follows that
  \[\|(E_\lambda-it)^{-1}\nu\|^2\geq \frac{1}{(|\lambda-it|+1)^2}\|\nu\|^2.\]
  With $u=\Real\lambda$, $v=\Imag\lambda$, we calculate
  \begin{align*}
    \int_{-\infty}^\infty\frac{dt}{(|\lambda-it|+1)^2}
    &\geq\int_{-\infty}^\infty\frac{dt}{2(|\lambda-it|^2+1)}
    =\frac{1}{2}\int_{-\infty}^\infty\frac{dt}{1+u^2+(t-v)^2}\\
    &=\frac{1}{2\sqrt{1+u^2}}\arctan\left(\frac{t-v}{\sqrt{1+u^2}}\right)
    \Big|_{t=-\infty}^\infty
    =\frac{\pi}{2\sqrt{1+u^2}}.
  \end{align*}
  Choosing $a>0$ such that $|\Real\lambda|\leq a$ for all 
  $\lambda\in\sigma_p(T)$, we obtain
  \[\int_{-\infty}^\infty\|(T-it)^{-1}x\|^2\,dt\geq m\frac{\pi}{2\sqrt{1+a^2}}
    \|\xi\|^2 \geq\frac{m\pi}{2M\sqrt{1+a^2}}\|x\|^2.\]
  \qedup
\end{proof}

A subspace $U\subset V$ of a Krein space is called \emph{nonnegative},
\emph{positive} and \emph{uniformly positive} if 
$\braket{x}{x}\geq 0$, $>0$
and $\geq\alpha\|x\|^2$, respectively, for all $x\in U\setminus\{0\}$, 
with some constant $\alpha>0$.
Nonpositive, negative and uniformly negative subspaces are defined accordingly.

In the context of dichotomous operators, the following result was obtained in
\cite{langer-tretter}.
\begin{prop}\label{prop:posinv}
  Let $T$ be a nonnegative Hamiltonian operator matrix with
  $\sigma_p^i(T)=\varnothing$, and consider the subspaces
  \[U_\pm=\overline{\sum_{\lambda\in\sigma_p^\pm(T)}\rsub(\lambda)}.\]
  Then $U_+$ is $J_2$-nonnegative and $U_-$ is
  $J_2$-nonpositive.

  If in addition $T$ is uniformly positive, has a Riesz basis of 
  Jordan chains, 
  each eigenvalue has finite multiplicity,
  and $\sigma_p(T)$ is contained in a strip around the imaginary axis,
  then $U_\pm$ is uniformly
  $J_2$-positive/\linebreak[0]-negative.
\end{prop}

\begin{proof}
  Let $W_\pm=\range(P_\pm)$ as in Lemma~\ref{lem:resolvint}.
  So $U_\pm=\overline{W_\pm}$.
  For $x\in W_+$,
  using the $J_2$-accretivity of $T$, we obtain
  \begin{align*}
    [x|x]&=\Real[P_+x-P_-x|x]=\frac{1}{\pi}\int_\bbR^\prime
    \Real[(T-it)^{-1}x|x]\,dt\\
    &=\frac{1}{\pi}\int_\bbR^\prime
    \Real[T(T-it)^{-1}x|(T-it)^{-1}x]\,dt
    \geq 0.
  \end{align*}
  Thus $W_+$ and hence also $U_+$ are nonnegative. 
  For $x\in W_-$ a similar calculation
  shows that $[x|x]\leq 0$ and hence $U_-$ is nonpositive.

  Now suppose that the additional assumptions on $T$ are satisfied.
  In particular, let  $B,C\geq\gamma>0$.
  For $x\in W_+$,
  using Lemma~\ref{lem:resolvint-estim}, we then 
  obtain
  \begin{align*}
    [x|x]&=\frac{1}{\pi}\int_\bbR^\prime
    \Real[T(T-it)^{-1}x|(T-it)^{-1}x]\,dt\\
    &\geq\frac{\gamma}{\pi}\int_\bbR\|(T-it)^{-1}x\|^2\,dt\geq
    \frac{\gamma c}{\pi}\|x\|^2.
  \end{align*}
  Consequently $U_+$ is uniformly positive.
  Again, a similar reasoning yields that $U_-$ is uniformly negative.
\end{proof}

\begin{remark}
  Proposition \ref{prop:ham-specsym} and
  Theorem~\ref{theo:hypmaxinv} also hold for 
  arbitrary (skew-) symmetric operators on Krein spaces since in the proofs
  the particular
  structure of the Hamiltonian as a block operator matrix was not used.
  Similarly, Proposition~\ref{prop:posinv} holds for arbitrary (uniformly)
  accretive  operators.
\end{remark}

\begin{lemma}\label{lem:hamposrootsub}
  Let $T$ be a nonnegative Hamiltonian operator matrix with 
  \begin{equation}
    C>0 \quad\text{and}\quad
    \ker(A^*-\lambda)\cap\ker B=\{0\}\quad\text{for all}\quad
    \lambda\in\bbC. \label{eq:hamposrootsub-c1}
  \end{equation}
  Then the root subspaces $\rsub(\lambda)$ of\, $T$ are $J_2$-positive 
  for $\Real\lambda>0$
  and $J_2$-negative for $\Real\lambda<0$.
\end{lemma}
\begin{proof}
  Suppose that  $\Real\lambda>0$; the proof for $\Real\lambda<0$ is
  analogous.
  From Proposition~\ref{prop:posinv} we know that $\rsub(\lambda)$ is 
  $J_2$-nonnegative.
  Let $x=(u,v)\in \rsub(\lambda)\setminus\{0\}$ and 
  $n\in\bbN$ minimal such that $(T-\lambda)^nx=0$.
  We use induction on $n$ to show that $[x|x]\neq0$ and thus $[x|x]>0$.
  
  For $n=1$ we have
  \[\Real\lambda\cdot[x|x]=\Real[Tx|x]=(Bv|v)+(Cu|u).\]
  If $[x|x]=0$, then $u=0$ since $B$ is nonnegative
  and $C$ positive. Hence
  \[Tx=\pmat{Bv\\-A^*v}=\lambda\pmat{0\\v},\]
  and \eqref{eq:hamposrootsub-c1} yields $v=0$, a contradiction.
  
  For $n>1$ we set $y=(T-\lambda)x$; so $[y|y]>0$ by the induction hypothesis. 
  If $[x|x]=0$, then
  \[0=\Real\lambda\cdot[x|x]=\Real[Tx|x]-\Real[y|x],\]
  i.e.,
  \[\Real[y|x]=(Bv|v)+(Cu|u)\geq 0.\]
  For $r\in\bbR$ let $w=rx+y$. Then $[w|w]=2r\Real[y|x]+[y|y]$.
  Since  $w\in \rsub(\lambda)$ is $J_2$-nonnegative and $r$ is arbitrary,
  this implies $\Real[y|x]=0$, i.e.\ $(Bv|v)+(Cu|u)=0$.
  So again $u=0$ and $(Bv|v)=0$.
  The reasoning from the proof of Proposition~\ref{prop:hamspecgap} then
  yields $Bv=0$.
  Consequently, the first component of $y$ is zero and hence
  $[y|y]=0$, again a contradiction.
\end{proof}

\section{Solutions of the Riccati equation}
\label{sec:sol}

In this section we consider Hamiltonian operator matrices which are
\emph{diagonally dominant},
i.e., $B$ and $C$ are relatively bounded to $A^*$ and $A$
respectively, see \cite{tretter-book};
in particular 
\begin{equation}\label{eq:ddomincl}
  \mdef(A)\subset\mdef(C), \qquad \mdef(A^*)\subset\mdef(B).
\end{equation}
Recall that, e.g., $C$ is \emph{relatively bounded} to $A$ if
$\mdef(A)\subset\mdef(C)$ and there are constants $a,b$ such that
$\|Cu\|\leq a\|u\|+b\|Au\|$ for all $u\in\mdef(A)$.
The infimum of all such $b$ is called the \emph{$A$-bound} of $C$.
Since for a Hamiltonian $T$ the operators $B$ and $C$ are symmetric and hence
closable, $T$ is diagonally dominant if $A$ is closed and
\eqref{eq:ddomincl} holds, see \cite[Remark~2.2.2]{tretter-book}.
In particular, the Hamiltonians from Theorem~\ref{theo:ham-invrbs}
and~\ref{theo:ham-rbeigvec} are 
diagonally dominant.

For an operator $X$ on the Hilbert space $H$ we consider the graph subspace
\[\Gamma(X)=\Bigl\{\pmat{u\\Xu}\,\Big|\,u\in\mdef(X)\Bigr\}.\]
It is well known that invariant graph subspaces of block operator
matrices are connected to Riccati equations. 
Here we have the following relations, see also
\cite[Section~4.3]{wyss-phd}:

\begin{prop}\label{prop:riccequiv}
  Let $T$ be a diagonally dominant Hamiltonian and $X$ an operator on $H$.
  \begin{itemize}
  \item[(i)] $\Gamma(X)$ is $T$-invariant
    if and only if\, $X$ satisfies the Riccati equation
    \begin{equation}\label{eq:req}
      X(Au+BXu)=Cu-A^*Xu \quad\text{for all}\quad
      u\in\mdef(A)\cap X^{-1}\mdef(A^*).
    \end{equation}
    (In particular $Au+BXu\in\mdef(X)$ for 
    $u\in\mdef(A)\cap X^{-1}\mdef(A^*)$.)
  \item[(ii)] If\, $T$ has a finitely spectral Riesz basis of subspaces
    $(V_k)_{k\in\bbN}$ and $\Gamma(X)$ is $T$-invariant and compatible
    with $(V_k)_{k\in\bbN}$, then $\mdef(A)\cap X^{-1}\mdef(A^*)$ is a core for
    $X$.
  \item[(iii)] If\, $X$ is selfadjoint and $\mdef(A)\cap X^{-1}\mdef(A^*)$
    is a core for $X$, then \eqref{eq:req} holds if and only if
    \begin{equation}\label{eq:weakreq}
      (Xu|Av)+(Au|Xv)+(BXu|Xv)-(Cu|v)=0
    \end{equation}
    for all $u,v\in\mdef(A)\cap X^{-1}\mdef(A^*)$.
  \end{itemize}
\end{prop}
\begin{proof}
  (i):  $\Gamma(X)$ is $T$-invariant if and only if for
  all $u\in\mdef(A)\cap\mdef(X)$ with $Xu\in\mdef(A^*)$ there exists
  $v\in\mdef(X)$ such that
  \[T\pmat{u\\Xu}=\pmat{Au+BXu\\Cu-A^*Xu}=\pmat{v\\Xv},\]
  and this is obviously equivalent to \eqref{eq:req}.

  (ii): By assumption, we have 
  $\Gamma(X)=\bigoplus^2_kU_k$ with $U_k\subset\mdef(T)$.
  Then $\sum_kU_k$ is dense in $\Gamma(X)$, and hence
  the subspace $D\subset H$ obtained by projecting $\sum_kU_k$ onto the first
  component is a core for $X$. Moreover
  $D\subset\mdef(A)\cap X^{-1}\mdef(A^*)$ since $\sum_kU_k\subset\mdef(T)$;
  hence $\mdef(A)\cap X^{-1}\mdef(A^*)$ is a core for $X$.

  (iii): Taking the scalar product of \eqref{eq:req} with 
  $v\in\mdef(A)\cap X^{-1}\mdef(A^*)$, we immediately get \eqref{eq:weakreq}.
  On the other hand, \eqref{eq:weakreq} can be rewritten as
  \[(Au+BXu|Xv)=(Cu-A^*Xu|v).\]
  Since $\mdef(A)\cap X^{-1}\mdef(A^*)$ is a core for $X$, 
  this equation holds for all $v\in\mdef(X)$.
  Consequently $Au+BXu\in\mdef(X^{*})=\mdef(X)$ and 
  \eqref{eq:req} follows.
\end{proof}

Graph subspaces are also naturally connected to the Krein space inner
products considered in Section~\ref{sec:ham}, see also
\cite{dijksma-desnoo}.
\begin{lemma}\label{lem:selfadjgraph}
  Consider an operator $X$ on the Hilbert space $H$.
  \begin{itemize}
  \item[(i)] $X$ is Hermitian, i.e.\
    $(Xu|v)=(u|Xv)$ for all $u,v\in\mdef(X)$, if and only if\,
    $\Gamma(X)$ is $J_1$-neutral.
  \item[(ii)] $X$ is selfadjoint if and only if\,
    $\Gamma(X)$ is hypermaximal $J_1$-neutral.
  \end{itemize}
  If\, $X$ is Hermitian, then
  \begin{itemize}
  \item[(iii)]  $X$ is nonnegative and nonpositive if and only if\,
    $\Gamma(X)$ is $J_2$-nonnegative and $J_2$-nonpositive, respectively;
  \item[(iv)] $X$ is bounded and uniformly positive (negative) if and only if\,
    $\Gamma(X)$ is uniformly $J_2$-negative (positive).
  \end{itemize}
\end{lemma}
\begin{proof}
  The assertions (i) and (iii) are immediate.
  For (ii) suppose
  $\Gamma(X)$ is hypermaximal $J_1$-neutral.
  If $w\in\mdef(X)^\perp$ then
  \[\Braket{\pmat{u\\X u}}{\pmat{0\\w}}=i(u|w)=0
    \quad\text{for all }u\in\mdef(X).\]
  Hence $(0,w)\in \Gamma(X)^\Jorth=\Gamma(X)$ and so $w=0$; $X$ is 
  densely defined. Since $X$ is also Hermitian, it is thus symmetric,
  $X\subset X^*$.
  If now $v\in\mdef(X^*)$, then
  \[\Braket{\pmat{u\\X u}}{\pmat{v\\X^* v}}
    =i(u|X^* v)-i(X u|v)=0
    \quad\text{for all }u\in\mdef(X),\]
  which implies $(v,X^*v)\in \Gamma(X)$ and so $v\in\mdef(X)$ and
  $X^*v=X v$. $X$ is thus selfadjoint.
  The converse implication in (ii) is proved similarly.

  (iv): Let $X$ be Hermitian and $\Gamma(X)$ uniformly $J_2$-positive.
  Then
  \[2\|Xu\|\|u\|\geq 2(Xu|u)
    =\Bigl[\pmat{u\\Xu}\Big|\pmat{u\\Xu}\Bigr]
    \geq\alpha\Big\|\pmat{u\\Xu}\Big\|^2
    =\alpha\|u\|^2+\alpha\|Xu\|^2,\]
  implies that $(Xu|u)\geq\frac{\alpha}{2}\|u\|^2$ and
  $\|Xu\|\leq\frac{2}{\alpha}\|u\|$. The proof of the
  other assertions is similar.
\end{proof}

\begin{theo}\label{theo:riccsol}
  Let $T$ be a diagonally dominant, nonnegative Hamiltonian
  operator matrix with $\varrho(T)\cap i\bbR\neq\varnothing$ and
  a finitely spectral Riesz basis of subspaces
  $(V_k)_{k\in\bbN}$.
  Suppose that
  \begin{itemize}
  \item[(a)] $B$ is positive, or
  \item[(b)] there is a connected component $M$ of $\varrho(A)$ such that
    $M\cap\varrho(T)\cap i\bbR\neq\varnothing$ and
    \begin{equation}\label{eq:apcontr}
      \mspan\bigl\{(A-z)^{-1}B^*u\,\big|\,z\in M,\,u\in\mdef(B^*)\bigr\}
      \subset H \quad\text{is dense}.
    \end{equation}
  \end{itemize}
  Then every hypermaximal $J_1$-neutral, $T$-invariant, compatible subspace
  $U$ is the graph $U=\Gamma(X)$ of a selfadjoint operator $X$ satisfying
  the Riccati equation
  \begin{equation}\label{eq:ricceq}
    X(Au+BXu)=Cu-A^*Xu, \qquad u\in\mdef(A)\cap X^{-1}\mdef(A^*),
  \end{equation}
  and $\mdef(A)\cap X^{-1}\mdef(A^*)$ is a core for $X$.
\end{theo}

\begin{proof}
  In view of Proposition~\ref{prop:riccequiv} and Lemma~\ref{lem:selfadjgraph},
  we only need to show that $U$ is a graph subspace.
  For this it is sufficient that $(0,w)\in U$ implies $w=0$.
  Suppose (a) holds and let $it\in\varrho(T)$, $t\in\bbR$.
  Let $(0,w)\in U$ and set $(u,v)=(T-it)^{-1}(0,w)$. Then
  \[(A-it)u+Bv=0,\quad Cu-(A^*+it)v=w.\]
  Since $U$ is $J_1$-neutral and invariant under $(T-it)^{-1}$,
  this implies
  \[0=\Braket{\pmat{0\\w}}{\pmat{u\\v}}=-i(w|u)\]
  and thus
  \[0=(w|u)=(Cu|u)-(v|(A-it)u)=(Cu|u)+(Bv|v).\]
  Since $B$ is positive and $C$ nonnegative, this implies $v=0$, and 
  the reasoning from the proof of Proposition~\ref{prop:hamspecgap}
  also yields $Cu=0$.
  Hence $w=0$.

  In the case of (b), for $it\in M\cap\varrho(T)\cap i\bbR$ we consider
  $u,v$ as above and obtain now $Cu=Bv=0$.
  Since $it\in\varrho(A)$, we have $-it\in\varrho(A^*)$.
  For $\tilde{u}\in\mdef(B^*)$ we get
  \[\bigl((A^*+it)^{-1}w\big|B^*\tilde{u}\bigr)=-(v|B^*\tilde{u})
    =-(Bv|\tilde{u})=0.\]
  Consequently, the function
  \(f(z)=((A^*-\bar{z})^{-1}w|B^*\tilde{u})\),
  which is holomorphic on $M$,
  vanishes on $M\cap\varrho(T)\cap i\bbR$.
  From the identity theorem we thus obtain
  \[0=\bigl((A^*-\bar{z})^{-1}w\big|B^*\tilde{u}\bigr)
    =\bigl(w\,\big|\,(A-z)^{-1}B^*\tilde{u}\bigr)
    \quad\text{for all } z\in M,\]
  and \eqref{eq:apcontr} now implies $w=0$.
\end{proof}

\begin{remark}\label{rem:altricc}
  Applying the previous theorem to the Hamiltonian
  \begin{equation}\label{eq:altham}
    \widetilde{T}=\pmat{-A^*&C\\B&A}
    =\pmat{0&I\\I&0}\pmat{A&B\\C&-A^*}\pmat{0&I\\I&0},
  \end{equation}
  we immediately get the following symmetric statement:
  If $C$ is positive or there is a connected component $M$ of
  $\varrho(A)$ such that
  $M\cap\varrho(T)\cap i\bbR\neq\varnothing$ and
  \begin{equation}\label{eq:apobsv}
    \mspan\bigl\{(A^*-\bar{z})^{-1}C^*v\,\big|\,z\in M,\,v\in\mdef(C^*)
    \bigr\} \subset H \quad\text{is dense},
  \end{equation}
  then a hypermaximal $J_1$-neutral, $T$-invariant, compatible subspace $U$
  is the ``inverse'' graph
  \[U=\Gamma_\mathrm{inv}(Y)=\Bigl\{\pmat{Yv\\v}\,\Big|\,v\in\mdef(Y)\Bigr\}\]
  of a selfadjoint operator $Y$ such that
  \[Y(CYv-A^*v)=AYv+Bv, \qquad v\in\mdef(A^*)\cap Y^{-1}\mdef(A),\]
  and $\mdef(A^*)\cap Y^{-1}\mdef(A)$ is a core for $Y$.
  In particular, if simultaneously $U=\Gamma(X)=\Gamma_\mathrm{inv}(Y)$,
  then $X$ is injective and $X^{-1}=Y$.
\end{remark}

For bounded $B,C$, conditions analogous to \eqref{eq:apcontr}
and \eqref{eq:apobsv} have been used in \cite{langer-ran-rotten}.
In that setting, they are equivalent to the approximate controllability
of the pair $(A,B)$
and the approximate observability of $(A,C)$, respectively.
Here we have the following relation:
\begin{prop}\label{prop:apcontr-equivcond}
  Let $A,B$ be densely defined operators on a Hilbert space $H$
  and $M\subset\varrho(A)$.
  Then for the assertions
  \begin{itemize}
  \item[(i)] $\mspan\bigl\{(A-z)^{-1}B^*v\,\big|\,z\in M,v\in\mdef(B^*)
    \bigr\}\subset H$ dense,
  \item[(ii)] $\ker(A^*-\lambda)\cap\ker B=\{0\}\quad\text{for all}\quad
    \lambda\in\bbC$,
  \end{itemize}
  we have the implication $(i)\Rightarrow(ii)$. 
  If\, $A$ is normal with compact resolvent,
  $\mdef(A)\subset\mdef(B)$, 
  and $M$ has an accumulation point in $\varrho(A)$, 
  then $(i)\Leftrightarrow(ii)$.
\end{prop}
\begin{proof}
For (i)$\Rightarrow$(ii) consider $A^*u=\lambda u$, $Bu=0$. Then
\[((A-z)^{-1}B^*v|u)=(v|B(A^*-\bar{z})^{-1}
  u)=(v|(\lambda-\bar{z})^{-1}Bu)=0\]
for every $z\in M$, $v\in\mdef(B^*)$ and (i) implies $u=0$. 

Now let $A$ be normal with compact resolvent.
Let $(\lambda_k)_{k\in\bbN}$ be the eigenvalues of $A$
and $P_k$ the corresponding orthogonal projections onto the eigenspaces.
To prove (i), let $u\in H$ be such that
$((A-z)^{-1}B^*v|u)=0$ for all $z\in M$, $v\in\mdef(B^*)$;
we aim to show $u=0$.
The function
\[f(z)=\bigl((A-z)^{-1}B^*v\big|u\bigr)=\sum_{k=0}^\infty\frac{1}{\lambda_k-z}
  (P_kB^*v|u)\]
is holomorphic on $\varrho(A)$ and vanishes on $M$;
hence $f=0$ by the identity theorem.
If we integrate the series along a circle in $\varrho(A)$ enclosing 
exactly one $\lambda_k$, we obtain
\[0=(P_kB^*v|u)=(B^*v|P_ku) \quad\text{for all}\quad
  v\in\mdef(B^*),\]
i.e.\ $P_ku\in\range(B^*)^\perp=\ker \overline{B}$.
Since $P_ku\in\mdef(A)\subset\mdef(B)$, we have in fact $P_ku\in\ker B$.
Since the eigenspaces of $A$ and $A^*$ coincide, (ii) now implies
$P_ku=0$ for all $k\in\bbN$ and thus $u=0$.
\end{proof}

\begin{coroll}\label{coroll:nonnegsol}
  In the situation of Theorem~\ref{theo:riccsol}
  we have $\sigma_p^i(T)=\varnothing$ if and only if
  \[\ker(A-it)\cap\ker C=\{0\}\quad\text{for all}\quad t\in\bbR.\]
  In this case,
  for every sc-set $\sigma\subset\sigma_p(T)$ the associated compatible
  subspace $U_\sigma$ is hypermaximal $J_1$-neutral and thus
  $U_\sigma=\Gamma(X_\sigma)$ with a selfadjoint solution $X_\sigma$ of
  \eqref{eq:ricceq}.
  The solutions $X_\pm$ corresponding to $\sigma=\sigma_p^\pm(T)$ are
  nonnegative/nonpositive.

  If  $C$ is even positive,
  then every $X_\sigma$ is injective.
  In addition, $X_\pm$ is the uniquely determined
  nonnegative/nonpositive selfadjoint solution of \eqref{eq:ricceq} 
  whose graph is compatible with $(V_k)_{k\in\bbN}$.
\end{coroll}
\begin{proof}
  The characterisation of $\sigma_p^i(T)=\varnothing$ is immediate from
  Proposition~\ref{prop:hamspecgap} and~\ref{prop:apcontr-equivcond}.
  If $\sigma_p^i(T)=\varnothing$, then
  condition \eqref{eq:hypmaxcond}
  in Theorem~\ref{theo:hypmaxinv} is trivially satisfied and hence
  $U_\sigma$ is hypermaximal $J_1$-neutral.
  The subspace $U_\pm$ associated with $\sigma_p^\pm(T)$ is
  $J_2$-nonnegative/-nonpositive by Proposition~\ref{prop:posinv}
  and hence $X_\pm$ is nonnegative/nonpositive by
  Lemma~\ref{lem:selfadjgraph}.

  Now suppose that $C>0$. Then
  $X_\sigma$ is injective by Remark~\ref{rem:altricc}.
  Let $X$ be nonnegative selfadjoint and
  $\Gamma(X)=\bigoplus^2_kU_k$ with $U_k\subset V_k$ $T$-invariant.
  Then each $U_k$ is $J_2$-nonnegative and the span of certain root vectors
  of $T$.
  By Proposition~\ref{prop:apcontr-equivcond}, Lemma~\ref{lem:hamposrootsub}
  can be applied and yields that
  $U_k$ is the span of root vectors corresponding to
  eigenvalues in the right half-plane. Therefore $U_k\subset U_+$ and
  hence $\Gamma(X)\subset U_+$.
  Consequently $X\subset X_+$ and thus $X=X_+$ since both operators are
  selfadjoint.
  The proof of the uniqueness of $X_-$ is analogous.
\end{proof}

\section{Bounded solutions}
\label{sec:bndsol}

Consider a diagonally dominant Hamiltonian $T$ and the decomposition
\begin{equation}\label{eq:hamdecomp}
  T=G+S, \quad G=\pmat{A&0\\0&-A^*},\quad S=\pmat{0&B\\C&0}.
\end{equation}
\begin{definition}
  We say that $T$ is \emph{r0-diagonally dominant} (\emph{r} stands
  for resolvent)
  if there is a sequence $(z_k)$ in $\varrho(G)$ such that 
  \[\lim_{k\to\infty}\|S(G-z_k)^{-1}\|=0.\]
\end{definition}

\begin{lemma}\label{lem:a0dd}
  \begin{itemize}
  \item[(i)] If\, $T$ is r0-diagonally dominant,
    then $S$ is relatively bounded to $G$ with $G$-bound $0$.
    Moreover
    \[z_k\in\varrho(T) \qquad\text{with}\qquad
      (T-z_k)^{-1}=(G-z_k)^{-1}\bigl(I+S(G-z_k)^{-1}\bigr)^{-1}\]
    whenever $\|S(G-z_k)^{-1}\|<1$; 
    in particular $\varrho(T)\neq\varnothing$.
  \item[(ii)] If\, $S$ is relatively bounded to $G$ with $G$-bound $0$,
    and there is a sequence $(z_k)$ in $\varrho(G)$ and a constant
    $c>0$ such that
    \[\lim_{k\to\infty}|z_k|=\infty \qquad\text{and}\qquad
      \|(G-z_k)^{-1}\|\leq\frac{c}{|z_k|},\]
    then $T$ is  r0-diagonally dominant.
  \end{itemize}
\end{lemma}
\begin{proof}
  (i) is a consequence of the estimate
  \[\|Sx\|\leq\|S(G-z_k)^{-1}\|\|(G-z_k)x\|
    \leq\|S(G-z_k)^{-1}\|\bigl(\|Gx\|+|z_k|\|x\|\bigr)\]
  and a Neumann series argument.
  (ii) follows from
  \[\|G(G-z_k)^{-1}\|=\|I+z_k(G-z_k)^{-1}\|\leq 1+c\]
  and
  \[\|S(G-z_k)^{-1}\|\leq a\|(G-z_k)^{-1}\|+b\|G(G-z_k)^{-1}\|
    \leq\frac{ac}{|z_k|}+b(1+c),\]
  where $b>0$ can be chosen arbitrarily small.
\end{proof}

Since $p$-subordination with $p<1$ implies relative
boundedness with relative bound $0$, see e.g.\ 
\cite[Section~3.2]{wyss-phd},
the previous lemma yields that the Hamiltonians from 
Theorem~\ref{theo:ham-invrbs} and~\ref{theo:ham-rbeigvec} are
r0-diagonally dominant.

\begin{prop}\label{prop:bndreq}
  Let $T$ be an r0-diagonally dominant Hamiltonian
  and $X:H\to H$ bounded such that $\Gamma(X)$ is $T$- and 
  $(T-z)^{-1}$-invariant for all $z\in\varrho(T)$. 
  Then $X\mdef(A)\subset\mdef(A^*)$ and 
  \begin{equation}\label{eq:prop:bndreq}
  A^*Xu+XAu+XBXu-Cu=0, \qquad u\in\mdef(A).
  \end{equation}
  Moreover
  \[\sigma(A+BX)=\sigma(T|_{\Gamma(X)}), \quad
    \sigma_p(A+BX)=\sigma_p(T|_{\Gamma(X)}),\]
  and for every $\lambda\in\sigma_p(A+BX)$ the root subspace of\, $A+BX$
  corresponding to $\lambda$ is the projection onto
  the first component of the root subspace
  of\, $T|_{\Gamma(X)}$ corresponding to $\lambda$.
\end{prop}
\begin{proof}
  We consider the isomorphism $\varphi$ and the projection $\proj_1$
  given by
  \begin{equation*}
    \begin{aligned}
      \varphi:H&\to\Gamma(X),\\ u&\mapsto(u,Xu),
    \end{aligned}\qquad\text{and}\qquad
    \begin{aligned}
      \proj_1:H\times H&\to H,\\ (u,v)&\mapsto u.
    \end{aligned}
  \end{equation*}
  Hence $\varphi^{-1}=\proj_1|_{\Gamma(X)}$.
  Using the decomposition \eqref{eq:hamdecomp} and
  writing $E=\varphi^{-1}T|_{\Gamma(X)}\varphi$ and $F=\proj_1S\varphi$,
  we have
  \[E-F=\proj_1T\varphi-\proj_1S\varphi=\proj_1G\varphi
    =A|_{\mdef(A)\cap X^{-1}\mdef(A^*)}.\]
  The  $(T-z)^{-1}$-invariance of $\Gamma(X)$ implies that
  $\varphi^{-1}(T-z)^{-1}\varphi=(E-z)^{-1}$ for
  $z\in\varrho(T)$.
  Since $T$ is r0-diagonally dominant,
  we can now find $z\in\varrho(G)\cap\varrho(T)$ such that
  \begin{align*}
    F(E-z)^{-1}&=\proj_1S\varphi\circ\varphi^{-1}(T-z)^{-1}\varphi
    =\proj_1S(T-z)^{-1}\varphi\\
    &=\proj_1S(G-z)^{-1}\bigl(I+S(G-z)^{-1}\bigr)^{-1}\varphi
  \end{align*}
  and $\|F(E-z)^{-1}\|<1$.
  Consequently $z\in\varrho(E-F)=\varrho(A|_{\mdef(A)\cap X^{-1}\mdef(A^*)})$.
  Since also $z\in\varrho(A)$, we obtain
  $\mdef(A)\cap X^{-1}\mdef(A^*)=\mdef(A)$, i.e.\
  $X\mdef(A)\subset\mdef(A^*)$. 
  The Riccati equation \eqref{eq:prop:bndreq}
  then follows from \eqref{eq:req}.
  Moreover, we have
  \[\varphi^{-1}T|_{\Gamma(X)}\varphi=A+BX,\]
  which immediately implies the equality of the spectra and point spectra
  of $T|_{\Gamma(X)}$ and $A+BX$, and that $\varphi$ maps the root subspaces
  of $A+BX$ bijectively onto the corresponding ones of $T|_{\Gamma(X)}$.
\end{proof}

\begin{theo}\label{theo:bndsol-charac}
  Let $T$ be an r0-diagonally dominant Hamiltonian with
  compact resolvent and
  a finitely spectral Riesz basis of subspaces $(V_k)_{k\in\bbN}$.
  Let $X:H\to H$ be bounded.
  Then $\Gamma(X)$ is $T$-invariant and 
  compatible with $(V_k)_{k\in\bbN}$ if and only if 
  $X\mdef(A)\subset\mdef(A^*)$ and $X$ is a solution of the 
  Riccati equation
  \begin{equation}\label{eq:bndreq}
  A^*Xu+XAu+XBXu-Cu=0, \qquad u\in\mdef(A).
  \end{equation}
\end{theo}

\begin{proof}
  If $\Gamma(X)$ is invariant and compatible, then the assertion follows
  from Proposition~\ref{prop:bndreq}.
  So suppose that $X\mdef(A)\subset\mdef(A^*)$ and that \eqref{eq:bndreq}
  holds. In view of 
  Proposition~\ref{prop:invcomp} it suffices to find $z\in\varrho(T)$ such 
  that $\Gamma(X)$ is $(T-z)^{-1}$-invariant.
  Let $\varphi$ and $\proj_1$ be as above.
  Let $z\in\varrho(G)$, in particular $z\in\varrho(A)$.
  Since $X\mdef(A)\subset\mdef(A^*)$, we have 
  \[A-z=\proj_1(G-z)\varphi.\]
  Set $W=(G-z)\varphi(\mdef(A))$.
  Then $\proj_1$ maps $W$ bijectively onto $H$ and we have
  \[(A-z)^{-1}=\varphi^{-1}(G-z)^{-1}(\proj_1|_W)^{-1}.\]

  We want to show that $W$ is closed. Let $x_n\in W$ with $x_n\to x$
  as $n\to\infty$ and set $y_n=(G-z)^{-1}x_n$.
  Then $y_n\to (G-z)^{-1}x$ as well as
  \[y_n=\varphi(A-z)^{-1}\proj_1x_n\to\varphi(A-z)^{-1}\proj_1x.\]
  Consequently $(G-z)^{-1}x=\varphi(A-z)^{-1}\proj_1x$ and hence
  $x\in W$. The open mapping theorem now implies that
  $(\proj_1|_W)^{-1}$ is bounded.
  Since
  \[BX(A-z)^{-1}=\proj_1S\varphi \circ 
    \varphi^{-1}(G-z)^{-1}(\proj_1|_W)^{-1}
    =\proj_1S(G-z)^{-1}(\proj_1|_W)^{-1}\]
  and due to the r0-diagonally dominance of $T$, we can
  find $z\in\varrho(G)\cap\varrho(T)$ such that
  $\|BX(A-z)^{-1}\|<1$, which in turn yields $z\in\varrho(A+BX)$.
  Since \eqref{eq:bndreq} holds, $\Gamma(X)$ is $T$-invariant and
  $\varphi^{-1}T|_{\Gamma(X)}\varphi=A+BX$; in particular
  $\varrho(T|_{\Gamma(X)})=\varrho(A+BX)$.
  We end up with $z\in\varrho(T)\cap\varrho(T|_{\Gamma(X)})$, which implies
  that $\Gamma(X)$ is $(T-z)^{-1}$-invariant.
\end{proof}

\begin{remark}\label{rem:weakbndreq}
  Let $X$ be bounded and selfadjoint.
  Then $X\mdef(A)\subset\mdef(A^*)$ and
  \[A^*Xu+XAu+XBXu-Cu=0, \qquad u\in\mdef(A),\]
  if and only if $X\mdef(A)\subset\mdef(B)$ and
  \[(Xu|Av)+(Au|Xv)+(BXu|Xv)-(Cu|v)=0, \quad u,v\in\mdef(A).\]
  Indeed, the second equation implies that $(Xu|Av)$ is bounded in $v$;
  hence $Xu\in\mdef(A^*)$ and the first equation follows.
\end{remark}

\begin{lemma}\label{lemma:splitXbound}
  Let $X_+$, $X_-$ be bounded selfadjoint operators on a
  Hilbert space $H$ with $X_+$ uniformly positive and $X_-$ nonpositive.
  If\, $X$ is a Hermitian operator on $H$ satisfying
  $\mdef(X)=D_+\dotplus D_-$, $X|_{D_\pm}=X_\pm|_{D_\pm}$,
  then $X$ is bounded.
\end{lemma}
\begin{proof}
  First consider $u\in D_+$, $v\in D_-$ with 
  $\|u\|=\|v\|=1$. Then
  \begin{align*}
    \Real(u-v|X_+u+X_-v)&=\Real\bigl((u|X_+u)-(v|Xu)+(u|Xv)-(v|X_-v)\bigr)\\
    &=(u|X_+u)-(v|X_-v)\geq\gamma
  \end{align*}
  where  $X_+\geq\gamma>0$ and hence
  \[\gamma\leq|(u-v|X_+u+X_-v)|\leq
    \|u-v\|\cdot\bigl(\|X_+\|+\|X_-\|\bigr).\]
  This implies
  \[1-\Real(u|v)=\frac{1}{2}\|u-v\|^2\geq\delta
    \quad\text{with}\quad \delta=\frac{1}{2}
    \left(\frac{\gamma}{\|X_+\|+\|X_-\|}\right)^2>0.\]
  Consequently
  \[|(u|v)|\leq 1-\delta\quad\text{for all}\quad u\in D_+,v\in D_-\text{ with }
  \|u\|=\|v\|=1.\]
  Now for arbitrary $u\in D_+$, $v\in D_-$ we have the estimates
  \begin{align*}
    &\|X(u+v)\|
    =\|X_+u+X_-v\|\leq\max\{\|X_+\|,\|X_-\|\}\bigl(\|u\|+\|v\|\bigr), \\
    &\bigl(\|u\|+\|v\|\bigr)^2\leq 2\bigl(\|u\|^2+\|v\|^2\bigr),\\
    &\begin{aligned}
      \|u+v\|^2&\geq\|u\|^2+\|v\|^2-2|(u|v)|
      \geq\|u\|^2+\|v\|^2-2(1-\delta)\|u\|\|v\|\\
      &\geq\|u\|^2+\|v\|^2-(1-\delta)\bigl(\|u\|^2+\|v\|^2\bigr)
      =\delta\bigl(\|u\|^2+\|v\|^2\bigr).\end{aligned}
  \end{align*}
  Therefore
  \[\|X(u+v)\|\leq\sqrt{\frac{2}{\delta}}\max\bigl\{\|X_+\|,\|X_-\|\bigr\}
    \|u+v\|,\]
  $X$ is bounded.
\end{proof}

Recall from Proposition~\ref{prop:hamspecgap} and \eqref{invrbs-resolv}
that a closed uniformly positive Hamiltonian with a finitely spectral
Riesz basis of subspaces satisfies
$\set{z\in\bbC}{|\Real z|<\gamma}\subset\varrho(T)$ for some $\gamma>0$.

\begin{theo}\label{theo:bndriccsol}
  Let $T$ be a uniformly positive, r0-diagonally dominant
  Hamiltonian with a Riesz basis of Jordan chains, where
  each eigenvalue has finite multiplicity and
  $\sigma_p(T)$ is contained in a strip around $i\bbR$.
  \begin{itemize}
  \item[(i)]
    If\, $U$ is a hypermaximal $J_1$-neutral, $T$-invariant,
    compatible subspace, then
    $U=\Gamma(X)$ where $X$ is bounded, selfadjoint, boundedly invertible,
    $X\mdef(A)=\mdef(A^*)$, and
    $X$ is a solution of the Riccati equation
    \begin{equation}\label{eq2:bndreq}
      A^*Xu+XAu+XBXu-Cu=0, \qquad u\in\mdef(A).
    \end{equation}
    Moreover, the solutions $X_\pm$ corresponding
    to the compatible subspaces $U_\pm$
    associated with $\sigma_p^\pm(T)$ are uniformly positive/negative and
    \begin{equation}\label{eq:riccincl}
      X_-\leq X\leq X_+, \qquad X_-^{-1}\leq X^{-1}\leq X_+^{-1}.
    \end{equation}
  \item[(ii)]
    If\, $X$ is a closed symmetric operator satisfying
    $\mdef(A)\subset\mdef(X)$, $X\mdef(A)\subset\mdef(B)$, and
    \begin{equation}\label{eq:weakbndreq}
      (Xu|Av)+(Au|Xv)+(BXu|Xv)-(Cu|v)=0, \quad u,v\in\mdef(A),
    \end{equation}
    then $X$ is bounded, $X\mdef(A)\subset\mdef(A^*)$ and 
    \eqref{eq2:bndreq} and the first inequality in \eqref{eq:riccincl} hold.
    If in addition $T$ has a compact resolvent, then 
    $\Gamma(X)$ is hypermaximal $J_1$-neutral,
    $T$-invariant and compatible, and hence all conclusions of (i) hold.
  \item[(iii)] If\, $X$ is bounded and $\Gamma(X)$ is $T$-invariant and
    compatible, then there exists a projection $P$ such that 
    \[X=X_+P+X_-(I-P).\]
  \end{itemize}
\end{theo}

\begin{proof}
  (i):
  Theorem~\ref{theo:riccsol} and Remark~\ref{rem:altricc} yield that $U$
  is a graph $U=\Gamma(X)$ with $X$ selfadjoint and injective.
  In particular $U_\pm=\Gamma(X_\pm)$ where $X_\pm$ is also bounded and
  uniformly positive/negative
  by Proposition~\ref{prop:posinv} and Lemma~\ref{lem:selfadjgraph}.
  Let $(\lambda_k)_{k\in\bbN}$ be the eigenvalues of $T$.
  Since the root subspaces $\rsub(\lambda_k)$ of $T$ form a Riesz basis, 
  we have 
  $\Gamma(X)=\bigoplus^2_{k\in\bbN}U_k$ with
  $T$-invariant subspaces $U_k\subset\rsub(\lambda_k)$. Hence
  \begin{equation}\label{eq:bndsoldec}
    \Gamma(X)= W_+\oplus W_- \quad\text{with}\quad
    W_+=\rbsdec_{\Real\lambda_k>0}U_k, \quad
    W_-=\rbsdec_{\Real\lambda_k<0}U_k,
  \end{equation}
  and $W_\pm\subset\Gamma(X_\pm)$.
  If $D_\pm=\proj_1(W_\pm)$ where $\proj_1$ is the projection onto
  the first component,
  then $\mdef(X)=D_+\dotplus D_-$, $X|_{D_\pm}=X_\pm|_{D_\pm}$,
  and Lemma~\ref{lemma:splitXbound} implies that $X$ is bounded.
  From Proposition~\ref{prop:bndreq} we thus obtain
  $X\mdef(A)\subset\mdef(A^*)$ and \eqref{eq2:bndreq}.
  Then also \eqref{eq:weakbndreq}, and the first inequality in
  \eqref{eq:riccincl} will be a consequence of (ii).
  As $\Gamma(X_{(\pm)})=\Gamma_\mathrm{inv}(X^{-1}_{(\pm)})$, 
  the above reasoning applied to the Hamiltonian
  $\widetilde{T}$ from \eqref{eq:altham}
  yields the boundedness of $X^{-1}$, $X^{-1}\mdef(A^*)\subset\mdef(A)$
  (hence $X\mdef(A)=\mdef(A^*)$),
  and the second inequality in \eqref{eq:riccincl}.

  (ii):
  Since equation \eqref{eq:weakbndreq} holds for $X_+$, we have
  \begin{align*}
    0&=(Au|(X_+-X)u)+((X_+-X)u|Au)+(BX_+u|X_+u)-(BXu|Xu)\\
    &=((A+BX_+)u|(X_+-X)u)+((X_+-X)u|(A+BX_+)u)\\
    &\hspace{45ex}-(B(X_+-X)u|(X_+-X)u)
  \end{align*}
  for $u\in\mdef(A)$.
  With $\Delta=X_+-X$ and $t\in\bbR$ we obtain
  \[2\Real\bigl((A+BX_+-it)u\big|\Delta u\bigr)
    =(B\Delta u|\Delta u)\geq 0.\]
  As a consequence of Proposition~\ref{prop:bndreq}, we have
  that $i\bbR\subset\varrho(A+BX_+)$, that $\sigma_p(A+BX_+)$ is contained
  in the right half-plane,
  and that the system of root subspaces $(L_\lambda)$ of $A+BX_+$ is complete
  in $H$.
  Then
  \[\Real\bigl( v\big|\Delta(A+BX_+-it)^{-1}v\bigr)\geq 0
    \quad\text{for}\quad v\in H,\]
  and Lemma~\ref{lem:resolvint} yields
  \[(\Delta v|v)=\frac{1}{\pi}\int_\bbR^\prime
    \Real\bigl(\Delta v\big|(A+BX_+-it)^{-1}v\bigr)\,dt\geq 0
    \quad\text{for}\quad 
    v\in\!\!\sum_{\lambda\in\sigma_p(A+BX_+)}\!\!L_\lambda.\]
  Hence $X\leq X_+$ on $\sum_\lambda L_\lambda$.
  Analogously we find $X_-\leq X$ on $\sum_\lambda L_\lambda$.
  Since $X_+$ and $X_-$ are bounded, this implies that 
  $X$ is bounded on $\sum_\lambda L_\lambda$ and hence on $H$ since $X$
  is closed.
  Consequently $X_-\leq X\leq X_+$ holds on $H$,
  and $X\mdef(A)\subset\mdef(A^*)$ and \eqref{eq2:bndreq} follow by
  Remark~\ref{rem:weakbndreq}.
  
  Let now $T$ have a compact resolvent.
  Theorem~\ref{theo:bndsol-charac} implies that $\Gamma(X)$ is a
  compatible subspace. It is also hypermaximal $J_1$-neutral since 
  $X$ is selfadjoint.

  (iii): We have again the decomposition \eqref{eq:bndsoldec}.
  In particular, $(U_k)$ is a Riesz basis of $\Gamma(X)$.
  Let $D_k=\proj_1(U_k)$.
  Then $(D_k)$ is complete in $H$.
  Moreover, if $c$ is the constant from \eqref{rbs-ineq-fin} for the basis 
  $(U_k)$ and $u_k\in D_k$, then
  \begin{gather*}
    c^{-1}\sum_{k=0}^n\|u_k\|^2
    \leq c^{-1}\sum_{k=0}^n\Bigl\|\pmat{u_k\\Xu_k}\Bigr\|^2
    \leq\Bigl\|\sum_{k=0}^n\pmat{u_k\\Xu_k}\Bigr\|^2
    \leq(1+\|X\|^2)\Bigl\|\sum_{k=0}^nu_k\Bigr\|^2,\\
    \Bigl\|\sum_{k=0}^nu_k\Bigr\|^2\leq\Bigl\|\sum_{k=0}^n\pmat{u_k\\Xu_k}
    \Bigr\|^2 \leq c\sum_{k=0}^n\Bigl\|\pmat{u_k\\Xu_k}\Bigr\|^2
    \leq c(1+\|X\|^2)\sum_{k=0}^n\|u_k\|^2.
  \end{gather*}
  So $(D_k)_{k\in\bbN}$ is a Riesz basis of subspaces of $H$.
  Consequently, we have the decomposition
  \[H=\rbsdec_{\Real\lambda_k>0}D_k\oplus\rbsdec_{\Real\lambda_k<0}D_k.\]
  Let $P:H\to H$ be the corresponding projection onto
  $\bigoplus^2_{\Real\lambda_k>0}D_k$.
  Since $X|_{D_k}=X_\pm|_{D_k}$ for $\Real\lambda_k\gtrless 0$,
  we obtain $X=X_+P+X_-(I-P)$.
\end{proof}

\section{Examples}
\label{sec:exam}

In the first example we consider a Hamiltonian for which the 
Riccati equation has unbounded solutions which
can be explicitly calculated.
In the other examples we apply our theory to non-trivial Riccati equations
involving differential operators.

\begin{example}\label{ex:unbndsol}
  Let $T$ be a nonnegative Hamiltonian such that $A$ is normal,
  $B=I$, $C$ is selfadjoint, and $A$ and $C$ admit an orthonormal basis
  $(e_k)_{k\geq 1}$ of common eigenvectors,
  $Ae_k=ik^2e_k$ and $Ce_k=ke_k$ for $k\geq 1$.
  Then $C$ is $1/2$-subordinate to $A$ and  Theorem~\ref{theo:ham-invrbs}
  can be applied.
  The subspaces  $V_k=\bbC e_k\times\bbC e_k$ constitute
  an orthogonal decomposition $H\times H=\bigoplus_kV_k$, which is
  obviously finitely spectral for $T$ with
  \[T|_{V_k}\cong\pmat{ik^2&1\\k&ik^2}.\]
  The eigenvalues and corresponding normalised eigenvectors of $T|_{V_k}$
  are
  \[\lambda_k^\pm=ik^2\pm \sqrt{k},\quad
    v_k^\pm=\frac{1}{\sqrt{1+k}}\pmat{e_k\\\pm \sqrt{k}e_k}.\]
  The hypermaximal $J_1$-neutral compatible subspace
  corresponding to 
  an sc-set $\sigma\subset\sigma(T)$ is given by
  \[
    U_\sigma=\rbsdec_{k\geq 1}U_k \quad\text{with}\quad
    U_k=\begin{cases}\bbC v_k^+&\text{if }\, \lambda_k^+\in\sigma,\\
           \bbC v_k^- &\text{if }\, \lambda_k^-\in\sigma,
	\end{cases}
  \]
  and it is the graph $U_\sigma=\Gamma(X_\sigma)$ of a selfadjoint solution 
  $X_\sigma$  of \eqref{eq:ricceq},
  \[
    X_\sigma e_k=\Biggl\{\begin{aligned}
       \sqrt{k}\,e_k &\quad\text{if }\, \lambda_k^+\in\sigma,\\
       -\sqrt{k}\,e_k &\quad\text{if }\, \lambda_k^-\in\sigma.
		 \end{aligned}
  \]
  In particular, $X_\sigma$ is unbounded and boundedly invertible.
  Consider now the sequences $(x_k)_{k\in\bbN}$, $(x_k^+)_{k\in\bbN}$
  and $(x_k^-)_{k\in\bbN}$ given by
  \[x_k=\pmat{\frac{2}{\sqrt{k}}e_k\\0}, \qquad
    x_k^\pm=\sqrt{\frac{1+k}{k}}\,v_k^\pm.\]
  Then $x_k=x_k^++x_k^-$ with $x_k^\pm\in\bbC v_k^\pm$,
  the sequence $(x_k)$ converges to zero, while the sequences
  $(x_k^\pm)$ do not.
  Consequently, the algebraic direct sum
  \[\bigoplus_{k\geq 1}\bbC v_k^+\dotplus\bigoplus_{k\geq 1}\bbC v_k^-\]
  is not topological direct, the system of eigenvectors
  $(v_k^\pm)_{k\geq 1}$ is not a Riesz basis,
  and 
  the operator $T$ is neither Riesz-spectral nor dichotomous, 
  see also Remark~\ref{rem:dichot}.
\end{example}

By choosing different eigenvalues for the operators $A$ and $C$ in
the previous example, it is easy to construct solutions $X_\sigma$ with 
different properties, for example solutions which are unbounded and
not boundedly invertible.

\begin{example}\label{ex:req-unbnddiff}
  Let $H=L^2([a,b])$ and consider the operators $A$, $B$, $C$ on $H$ 
  given by
  \begin{align*}
    &Au=u''',\quad Bu=-(g_1u')'+h_1u,\quad Cu=-(g_2u')'+h_2u,\\
    &\mdef(A)=\bigl\{u\in\ W^{3,2}([a,b])\,\big|\,
      u(a)=u(b)=0,\,u'(a)=u'(b)\bigr\},\\
    &\mdef(B)=\mdef(C)=\bigl\{u\in C^2([a,b])\,\big|\,
      u(a)=u(b)=0\bigr\}
  \end{align*}
  where $g_1,g_2\in C^1([a,b])$, $h_1,h_2\in L^2([a,b])$, 
  $g_1,g_2,h_1,h_2\geq 0$, and $W^{k,2}([a,b])$ denotes the Sobolev space
  of $k$ times weakly differentiable, square integrable functions.
  Then $A$ is skew-selfadjoint with compact resolvent,
  $0\in\varrho(A)$, and $\sigma(A)$ consists of at most two sequences
  of eigenvalues
  \[\lambda_{jk}=c_{jk}k^3, \quad k\geq k_{j0},\, j=1,2,\]
  with converging sequences $(c_{jk})$, see \cite{naimark}.
  Since the multiplicity of every eigenvalue is at most three,
  this implies that
  \[\sup_{r\geq1}\frac{N(r,A)}{r^{1/3}}<\infty.\]
  The operators $B$ and $C$ are symmetric and nonnegative.  
  Using Sobolev and interpolation inequalities, see \cite{adams},
  we can find constants $b_1,b_2,b_3\geq0$ such that
  \begin{align*}
    \|Bu\|_{L^2} 
    &\leq\|g_1\|_\infty\|u''\|_{L^2}+\|g_1'\|_{L^2}\|u'\|_\infty
    +\|h_1\|_{L^2}\|u\|_\infty
    \leq b_1\|u\|_{W^{2,2}}\\
    &\leq b_2\|u\|_{L^2}^{1/3}\|u\|_{W^{3,2}}^{2/3}
    \leq b_3\|u\|_{L^2}^{1/3}\bigl(\|u\|_{L^2}+\|u'''\|_{L^2}\bigr)^{2/3}\\
    &\leq b_3\bigl(\|A^{-1}\|+1\bigr)^{2/3}\|u\|_{L^2}^{1/3}
    \|Au\|_{L^2}^{2/3}
  \end{align*}
  for $u\in\mdef(A)$.
  Hence $B$, and similarly $C$, are $2/3$-subordinate to $A$.
  By Theorem~\ref{theo:ham-invrbs}, the Hamiltonian
  corresponding to $A,B,C$ thus has a finitely spectral Riesz basis of 
  subspaces.
  If $g_{1}>0$ or $h_{1}>0$, and if $g_2>0$ or $h_2>0$, then both
  $B$ and $C$ are positive,
  and Corollary~\ref{coroll:nonnegsol} yields
  an injective selfadjoint solution $X_\sigma$ of \eqref{eq:ricceq}
  for every sc-set $\sigma\subset\sigma(T)$.
\end{example}

The example above immediately generalises to normal differential operators
$A$ on $[a,b]$ of order $n$ and nonnegative symmetric differential operators
$B,C$ of order at most $n-1$.

\begin{example}\label{ex:req-uposmult}
  Let $H=L^2([-1,1])$ and consider the operators
  \begin{align*}
    &Au=u',\quad\mdef(A)=\bigl\{u\in W^{1,2}([-1,1])\,\big|\,u(-1)=u(1)
      \bigr\},\\
    &Bu=bu,\quad Cu=cu,\quad\mdef(B)=\mdef(C)=H
  \end{align*}
  with $b,c\in L^\infty([-1,1])$ and $b(t),c(t)\geq\gamma>0$ for
  almost all $t\in[-1,1]$. $A$ is 
  skew-selfadjoint with compact resolvent and simple eigenvalues
  $\lambda_k=i\pi k$.
  $B$ and $C$ are bounded and uniformly positive.
  If now $\|b\|_\infty,\|c\|_\infty<\pi/2$, then we can apply
  Theorems~\ref{theo:ham-rbeigvec} and~\ref{theo:bndriccsol}
  and obtain bounded, selfadjoint, boundedly
  invertible solutions of the Riccati equation \eqref{eq2:bndreq}.

  Consider now the special case that $c=\chi^2 b$ with
  \[\chi(t)=\begin{cases}1, &t<0,\\\alpha, &t\geq0,\end{cases}
    \qquad\alpha\in\bbR\setminus\{0,1\}.\]
  Let $X\in L(H)$ be the operator of multiplication with $\chi$.
  It is not hard to see that
  \[\mdef(A)\cap X^{-1}\mdef(A)
    =\Set{u\in W^{1,2}([-1,1])}{u(-1)=u(0)=u(1)=0}\]
  and $AXu=\chi u'$ for $u\in\mdef(A)\cap X^{-1}\mdef(A)$. Hence
  \[-AXu+XAu+XBXu-Cu
    =-\chi u'+\chi u'+\chi^2bu-cu=0.\]
  Consequently, $X$ is a solution of the Riccati equation
  \[-AXu+XAu+XBXu-Cu=0,\qquad u\in\mdef(A)\cap X^{-1}\mdef(A),\]
  and $\mdef(A)\cap X^{-1}\mdef(A)\subset H$ is dense.
  In particular, $\Gamma(X)$ is a $T$-invariant subspace.
  On the other hand, since $\mdef(A)\cap X^{-1}\mdef(A)\neq\mdef(A)$
  we have $X\mdef(A)\not\subset\mdef(A)$,
  and with Theorem~\ref{theo:bndsol-charac} we conclude that $\Gamma(X)$ is
  not a compatible subspace.
\end{example}

\section*{Acknowledgement}
The author is grateful for the support of
Deutsche Forschungsgemeinschaft DFG, grant no.\ TR368/6-1, and
of Schweizerischer Nationalfonds SNF, grant no.~15-486.

\bibliographystyle{cwyss}
\bibliography{biblist}

\end{document}